%% file: main.tex
\documentclass[a4paper,11pt]{article} %,draft
\usepackage{amsmath,amssymb,enumerate,color}
\usepackage{amsthm,color}
\usepackage{txfonts}
\usepackage{bigints}
\usepackage{comment}
\usepackage{graphicx}

\newcommand{\R}{\mathbb{R}}

\newcommand{\Z}{\mathbb{Z}}
\newcommand{\N}{\mathbb{N}}

\DeclareMathOperator{\supp}{supp}

\newtheorem{theorem}{Theorem}[section]
\newtheorem{lemma}[theorem]{Lemma}
\newtheorem{proposition}[theorem]{Proposition}
\newtheorem{corollary}[theorem]{Corollary}
\theoremstyle{remark}
\newtheorem{remark}{Remark}[section]
\theoremstyle{definition}

\newtheorem{definition}{Definition}[section]

\numberwithin{equation}{section}

\makeatletter
\def\@cite#1#2{[{{\bfseries #1}\if@tempswa , #2\fi}]}
\makeatother
\begin{document}
%\linenumbers
\begin{center}
\Large{{\bf
Small data blow-up for the weakly coupled system of the generalized Tricomi equations with multiple propagation speeds}}
\end{center}

\vspace{5pt}

\begin{center}
Masahiro Ikeda%
\footnote{
Department of Mathematics, Faculty of Science and Technology, Keio University, 3-14-1 Hiyoshi, Kohoku-ku, Yokohama, 223-8522, Japan/Center for Advanced Intelligence Project, RIKEN, Japan,
E-mail:\ {\tt masahiro.ikeda@keio.jp/masahiro.ikeda@riken.jp}},\
Jiayun Lin
\footnote{
Department of Mathematics and Science, School of Sciences, Zhejiang Sci-Tech University, 310018, Hangzhou, China,
E-mail:\ {\tt jylin@zstu.edu.cn}},\
Ziheng Tu
\footnote{
School of Data Science, Zhejiang University of Finance and Economics, 310018, Hangzhou, China,
E-mail:\ {\tt tuziheng@zufe.edu.cn}}
\end{center}

%%%%     ABSTRACT     %%%%%%
\newenvironment{summary}{\vspace{.5\baselineskip}\begin{list}{}{%
     \setlength{\baselineskip}{0.85\baselineskip}
     \setlength{\topsep}{0pt}
     \setlength{\leftmargin}{12mm}
     \setlength{\rightmargin}{12mm}
     \setlength{\listparindent}{0mm}
     \setlength{\itemindent}{\listparindent}
     \setlength{\parsep}{0pt}
     \item\relax}}{\end{list}\vspace{.5\baselineskip}}
\begin{summary}
{\footnotesize {\bf Abstract.}
In the present paper, we study the Cauchy problem for the weakly coupled system of the generalized Tricomi equations with multiple propagation speeds. Our aim of this paper is to prove a small data blow-up result and an upper estimate of lifespan of the problem for a suitable compactly supported initial data in the subcritical and critical cases of the Strauss type. The proof is based on the framework of the argument in the paper \cite{ISW19}. One of our new contributions is to construct two families of special solutions to the free equation (see (\ref{testsub}) or (\ref{testcri})) as the test functions and prove their several properties. We emphasize that the system with two different propagation speeds is treated in this paper and the assumption on the initial data is improved from the point-wise positivity to the integral positivity.
}
\end{summary}

{\footnotesize{\it Mathematics Subject Classification}\/ (2010): %
%Primary:
35B44; %	Blow-up in context of PDEs
35A01; %	Existence problems: global existence, local existence, non-existence
35L15;%	Initial value problems for second-order hyperbolic equations
35L05%	Wave equation
}
\tableofcontents

\input{Introduction.tex}
\input{Preliminaries.tex}
\input{Proofofthemain}

\section*{Acknowedgements}
The first author is supported by JST CREST Grant Number JPMJCR1913, Japan and Grant-in-Aid for Young Scientists Research (B) No.15K17571 and Young Scientists Research (No.19K14581), Japan Society for the Promotion of Science. This work was done when the second author was visiting Riken and Keio University, with the support of Chinese Scholarship Council. The second author is supported by NSFC No. 11501511 and Zhejiang Provincial Nature Science Foundation of China under Grant No. LQ15A010012

\input{thebibliography.tex}
\end{document}

%% file: Introduction.tex
\section{Introduction}
\subsection{Setting and backgrounds}
\ \ In the present paper, we study the Cauchy problem of the weakly coupled system of the generalized Tricomi equations with multiple propagation speed on the Euclidean space $\R^N$:
\begin{equation}\label{sys}
\left\{
\begin{array}{ll}
\partial_t^2u(x,t)-t^{m_1}\Delta u(x,t)=G_1(v(x,t)),\ &(x,t)\ \in\ \R^N\times[0,T),\\
\partial_t^2v(x,t)-t^{m_2}\Delta v(x,t)=G_2(u(x,t)),\ &(x,t)\ \in\ \R^N\times[0,T),\\
u(x,0)=\varepsilon f_1(x),\ \partial_tu(x,0)=\varepsilon g_1(x),\ &x\ \in\ \R^N,\\
v(x,0)=\varepsilon f_2(x),\ \partial_tv(x,0)=\varepsilon g_2(x),\ &x\ \in\ \R^N,\\
\end{array}
\right.
\end{equation}
where $N\in \mathbb{N}$ denotes the spatial dimension, $\partial_t:=\partial/\partial t$ is the time derivative, $\Delta:=\sum_{j=1}^N\partial^2/\partial_{x_j}^2$ is the Laplace operator on $\R^N$, $m_1, m_2\ge 0$ stands for the strength of the diffusion, $G_1:\mathbb{R}\rightarrow \mathbb{R}_{\ge 0}$ and $G_2:\mathbb{R}\rightarrow \mathbb{R}_{\ge 0}$ are nonlinear functions which satisfy $G_1\in C^1(\mathbb{R})$, $G_2\in C^1(\mathbb{R})$ and the assumption that there exist positive constants $a,b>0$ and $p,q>1$ such that the estimates
\begin{equation}
\label{nonlinearity}
  G_1(0)=0,\ \ G_2(0)=0, \ \ G_1(\sigma)\ge a|\sigma|^p,\ \ G_2(\sigma)\ge b|\sigma|^q
\end{equation}
hold for any $\sigma\in \R$, the pair of functions $(u,v):\mathbb{R}^N\times[0,T)\rightarrow\mathbb{R}^2$ is an unknown function, $f_1,f_2,g_1,g_2$ are prescribed functions describing the shape of the initial data, $\varepsilon>0$ is a small parameter showing the size of the initial data and $T=T(\varepsilon)>0$ denotes the maximal existence time of the function $(u,v)$, which is called lifespan (see Definition \ref{Life}). %Without loss of generality, we assume that $m_1\geq m_2$.

The significance to study the problem (\ref{sys}) mathematically comes from physical problems of Transonic gas dynamics (see \cite[Chapter 4]{B58} and \cite{Frankl} for more detail).

Our aim of the present paper is to prove a small data blow-up result and an upper estimate of lifespan of the problem (\ref{sys}) for suitable compactly supported initial data in the subcritical case and the critical case of the Strauss type, that is, the case $\Omega_{SS}(N,m_1,m_2,p,q)\ge 0$ for $m_1\neq m_2$ and $\Gamma_{SS}(N,m_1,m_2,p,q)\ge 0$ for $m_1 = m_2=m$, where $\Omega_{SS}$ and $\Gamma_{SS}$ are related to the so-called critical curves given by
\begin{align*}
\Omega_{SS}(N,m_1,m_2,p,q):&=
\begin{cases}
     F_{SS}(N,m_1,p,q),&\ m_1>m_2,\\
     F_{SS}(N,m_2,q,p),&\ m_1<m_2,
\end{cases}\\
\Gamma_{SS}(N,m,p,q):&=\max\{F_{SS}(N,m,p,q),F_{SS}(N,m,q,p)\}
\end{align*}
with a rational function $F_{SS}$ of $p,q$,
\[
F_{SS}(N,m,p,q):=-\frac{m+2}{4}(N-1)-\frac{m}{4}-\frac{m}{2q}+(pq-1)^{-1}\left(p+2+\frac 1q\right)
%-\frac{(m_1-m_2)N(q-1)}{2q(pq-1)}
\]
(see Remarks \ref{pq2} and \ref{p-qcurve} for the regions of $(p,q)$). Our result extends the previous results\cite[Theorem 1.1]{KO99}, \cite[Theorem 1]{AKH00}, \cite[Theorem 1.1]{KTW12}, \cite[Proposition 7.1]{ISW19} for the classical case $m_1=m_2=0$ to the general case $m_1,m_2\ge 0$. The proofs of the blow-up results are based on the framework of the argument in the paper \cite{ISW19} with appropriate test functions.
Our new contribution for the proof is to construct the test functions (\ref{testsub}) and (\ref{testcri}) and derive their several properties (Lemma \ref{propy} and Proposition \ref{propertest}). We also give an improvement of the previous blow-up results \cite[Theorem1.1]{He-Witt-Yin17} and \cite[Theorems 1.2 and 1.3]{LTpre} for the Cauchy problem of the single generalized Tricomi equation with the $\rho$-th order power nonlinearity
\begin{equation}
\label{single}
\begin{cases}
     \partial_t^2w(x,t)-t^m\Delta w(x,t)=|w(x,t)|^{\rho},&\ (x,t)\in \mathbb{R}^N\times [0,T),\\
     w(x,0)=\varepsilon f(x),\ \partial_tw(x,0)=\varepsilon g(x),&\
      x\in \mathbb{R}^N,
\end{cases}
\end{equation}
where $w=w(x,t):\mathbb{R}^N\times[0,T)\rightarrow \mathbb{R}$ is an unknown function, $m\ge 0$, $\rho>1$ is the exponent of the nonlinear term and $f,g\in C_0^{\infty}(\R^N)$ are given functions. More precisely, we relax the assumption of the positivity of the initial data $f$ and $g$ in a point-wise sense assumed in \cite[Theorem 1.1]{He-Witt-Yin17}, \cite[Theorem 1.1]{He-Witt-Yin17-2} and \cite[Theorems 1.2 and 1.3]{LTpre} and we only assume the inequality $\displaystyle{\int_{\R^N}g(x)dx>0}$ (see also \cite{ISW19} for the weakly coupled systems of semilinear wave equations, and \cite{SW20} for semilinear wave equation in an exterior domain).

\subsection{Known results}
\ \ \ We recall related results about the problem (\ref{sys}). There are many mathematical results about existence or non-existence of solution and properties of solution such as regularity and time decay to the problem (\ref{sys}) or (\ref{single}) (see \cite{BG99, BG02, He-Witt-Yin17, He-Witt-Yinpe, He-Witt-Yin17-2, RWY15-1, RWY15, RWY18, Y04, Yag06} and their references). The classical Tricomi equation, which was firstly investigated by Tricomi \cite{T23},
\begin{equation}
\label{Ctri}
\partial_t^2w(x,t)-t\partial_x^2w(x,t)=0,\ \ \ (x,t)\in \mathbb{R}\times\mathbb{R},
\end{equation}
is known as a typical second order partial differential equation of mixed type, which means that the cases $t<0$, $t=0$ and $t>0$ are elliptic, parabolic and hyperbolic, respectively. The fundamental solutions to (\ref{Ctri}) were computed explicitly in \cite{BG99, BG02}. Yagdjian \cite{Y04} studied the linear generalized Tricomi equation of the hyperbolic type with an external force $e:\mathbb{R}^N\times[0,\infty)\rightarrow\mathbb{R}$:
\begin{equation}
\label{inhomo}
    \partial_t^2w(x,t)-t^m\Delta w(x,t)=e(x,t),\ \ \ (x,t)\in \mathbb{R}^N\times [0,\infty),
\end{equation}
and derived an explicit formula of the fundamental solution to this problem. Moreover he proved $L^p$-$L^q$ time decay estimates for the solution. Ruan, Witt and Yin \cite{RWY15-1} proved local existence and singularity structures of solution with low regularity to the Cauchy problem of (\ref{inhomo}) with a smooth inhomogeneous term $e=e(t,x,w)$ and discountinuous initial data in the case of $N\ge 2$.

Next we recall results about the Cauchy problem of a semilinear equation with a $\rho$-th order power nonlinearity:
\begin{equation}
\label{nonli}
    \partial_t^2w(x,t)-t^m\Delta w(x,t)=\mathcal{N}(w(x,t)),\ \ \ (x,t)\in \mathbb{R}^N\times [0,T),
\end{equation}
where $\mathcal{N}:\mathbb{R}\rightarrow\mathbb{R}$ is a $\rho$-th order nonlinear function. The typical examples are $\mathcal{N}(z):=\pm|z|^{\rho-1}z$ and $\pm|z|^{\rho}$. D'Ancona \cite{D} obtained the existence of a smooth local solution to the Cauchy problem of (\ref{nonli}) with smooth data for the nonlinearity $\mathcal{N}(w(x,t))=-|w|^{\rho-1}w$ and $\rho<\frac{3m+10}{3m+2}$. Yagdjian \cite{Yag06} proved large data local well-posedness in the Lebesgue space $L^q(\R^N)$ for some $q>1$, and global existence or non-existence of $L^q$-solutions to the equation (\ref{nonli}) for some small initial data under some restrictions on the exponent $\rho$. Ruan, Witt and Yin \cite{RWY15} studied the Cauchy problem (\ref{nonli}) with $m=2l-1$ and $l\in \N$ on a mixed type domain $\R^N\times (-T,T)$ with $N\ge 2$ and proved existence of a local solution to (\ref{nonli}) for low regular initial data. They \cite{RWY18} showed local existence and uniqueness of a minimal regular solution to the Cauchy problem (\ref{nonli}) with $m\in \mathbb{N}$ and some $\rho\in \mathbb{N}$ on the hyperbolic domain $\mathbb{R}^N\times [0,T)$ with $N\ge 2$. He, Witt and Yin \cite{He-Witt-Yin17} studied existence of a global solution to the problem (\ref{single}) with $m\in \mathbb{N}$ and $N\ge 2$ and showed a small data blow-up result for suitable non-negative compactly supported smooth data $f,g\ge 0$ in the subcritical case $\rho\in (1,\rho_S)$, where $\rho_S=\rho_S(N,m)$ is called a critical exponent defined as the positive root of the quadratic equation of $\rho$,
\begin{equation}
\label{criticalexp}
\gamma_S(N,m,\rho):=-\left\{(m+2)\frac{N}{2}-1\right\}\rho^2-\left\{-(m+2)\frac{N}{2}+m-1\right\}\rho+(m+2)=0.
\end{equation}
We note that in the case of $m=0$, the quadratic polynomial $\gamma_S(N,0,\rho)$ is
\[
\gamma_S(N,0,\rho):=-(N-1)\rho^2+(N+1)\rho+2
\]
and the exponent $\rho_{S}(N,0)$ is
\[
    \rho_{S}(N,0)=
    \left\{\begin{array}{cc}
    \displaystyle{\frac{N+1+\sqrt{N^2+10N-7}}{2(N-1)}},& \text{if}\ N\ge 2,\\
    \infty,& \text{if}\ N=1,
    \end{array}\right.
\]
which is known as the Strauss exponent, and divides small data blow-up and small data global existence to the Cauchy problem (\ref{single}) with $m=0$. They \cite[Theorem 1.2]{He-Witt-Yin17} also proved small data global existence to the problem (\ref{single}) with $m\in \mathbb{N}$ for initial data $(f,g)\in H^s(\mathbb{R}^N)\times H^{s-\frac{2}{m+2}}(\mathbb{R}^N)$ with the minimal regularity $s=\frac{N}{2}-\frac{4}{(m+2)(\rho-1)}$ in the case where $\rho\in\left[ \rho_{\text{conf}},\infty\right]$ for $N=2$ or $\rho\in\left[\rho_{\text{conf}},\frac{(m+2)(n-2)+6}{(m+2)(n-2)-2}\right]$ for $N\ge 3$. Here $\rho_{\text{conf}}$ is called a conformal exponent given by $\rho_{\text{conf}}=\rho_{\text{conf}}(N,m):=\frac{(m+2)N+6}{(m+2)N-2}$. They \cite[Theorem 1.1]{He-Witt-Yinpe} proved small data global existence to the problem (\ref{single}) with $N\ge 3$ and $m\in \N$ and with smooth compactly supported data $f,g\in C_0^{\infty}(\R^N)$ in the supercritical case $\rho\in (\rho_{S},\rho_{\text{conf}})$. They \cite[Theorem 1.1]{He-Witt-Yin17-2} showed small data blow-up to the problem (\ref{single}) with $N\ge 2$ and $m\in \N$ and with non-negative smooth compactly supported data $f,g\in C_0^{\infty}(\R^N)$ in the critical case $\rho=\rho_{S}$ and they also showed small data global existence to (\ref{single}) with $N=2$ and $m\in \N$ and with low regular and weighted data in the supercritical case $\rho\in (\rho_S,\rho_{\text{conf}})$ (see \cite[Theorem 1.1]{He-Witt-Yin17-2}). The second and third authors \cite[Theorems 1.2 and 1.3]{LTpre} proved small data blow-up and an upper estimate of lifespan to the problem (\ref{single}) with $m\ge 0$ and $N\in\N$ and with suitable non-negative smooth compactly supported data $f,g\in C_0^{\infty}(\R^N)$ in the subcritical and critical cases $\rho\in (1,\rho_S]$, that is, the following estimate
\[
  T(\varepsilon)\le \left\{
  \begin{array}{cl}
  C \varepsilon^{-\frac {2\rho(\rho-1)}{\gamma_S(N,m,\rho)}},& \text{if}\ 1<\rho<\rho_{S},\\
\exp(C \varepsilon^{-\rho(\rho-1)}),& \text{if}\ \rho=\rho_{S},
\end{array}\right.
\]
holds for any $\varepsilon\in (0,\varepsilon_0]$, where $\varepsilon_0$ and $C$ are positive constants independent of $\varepsilon$. From the above results, we see that the exponent $\rho_S(N,m)$ at least in the case $m\in \Z_{\ge 0}$ is a critical exponent dividing global existence and blow-up to (\ref{single}) for small data, that is, the following statement holds:
\begin{equation}
\label{critical}
  \left\{
  \begin{array}{l}
  \text{if}\ \rho>\rho_S(N,m),\ \text{the solution exists globally in time for small data},\\
  \text{if}\ \rho<\rho_S(N,m),\ \text{the solution blows up in a finite time even for small data}.
\end{array}\right.
\end{equation}
The problem to determine the critical exponent to the semilinear classical ($m=0$) wave equation
\begin{equation}
\label{classi}
     \partial_t^2w(x,t)-\Delta w(x,t)=|w(x,t)|^{\rho},\ \ \ (x,t)\in \R^N\times[0,T),
\end{equation}
is well known as the Strauss conjecture and was solved (see \cite{YZ06, Z07} and their references). Also, roughly speaking, the sharp lower and upper bounds of lifespan in the case $N\ge 2$ to (\ref{classi}) were proved as
$$
T(\varepsilon)\sim \left\{
\begin{array}{ll}
C\varepsilon^{-\frac{p-1}{\gamma_G(N,0,\rho)}},&\quad \mbox{if}\ N=1\ \text{or}\ \rho\in (1,2),\\
Ca(\varepsilon),&\quad\mbox{if}\ N=\rho=2,\\
C \varepsilon^{-\frac {2\rho(\rho-1)}{\gamma_S(N,0,\rho)}},&\quad \mbox{if}\ N=2\ \text{and}\ \rho\in (2,\rho_S(N,0))\ \text{or}\ N\ge 3\ \text{and}\ \rho\in (1,\rho_S(N,0)),\\
\exp(C \varepsilon^{-\rho(\rho-1)}),&\quad \mbox{if}\ \rho=\rho_S(N,0),
\end{array}
\right.
$$
where $C>0$ is a positive constant independent of $\varepsilon$ and $a=a(\varepsilon)$ satisfies $\varepsilon^2a^2\log(1+a)=1$ (see \cite{TW11, ZH14, ISW19} and their references). Here for $N\in \N$ and $m\ge 0$, $\gamma_G(N,m,\rho)$ is defined as the following monomial of $\rho\ge 1$
\[
\gamma_{G}(N,m,\rho):=2-\left(\frac{m+2}{2}N-1\right)(\rho-1)
\]
and the root of the equation $\gamma_G(N,0,\rho)=0
$, that is, $\rho_G=\rho_G(N,0):=1+\frac{2}{N-1}$ is called the Glassey exponent in the research field of wave equations (see \cite{HWY12}). Our approach of this paper for the proof of our main result (Theorem \ref{Main}) is applicable to the single equation (\ref{single}) and can extend the result of the upper estimate to more general $m\ge 0$ except for the case $N=\rho=2$. Very recently, small data blow-up and upper estimate of lifespan for the generalized Tricomi equation (\ref{classi}) with a replacement of $|u|$ into the derivative nonlinearity $|\partial_tu|^p$ were studied in \cite{LPpre,LSpre}.

Determining the critical $(p.q)$-curve for the weakly coupled system of the classical semilinear wave equations was also studied in the case $N\ge 2$
\begin{equation}\label{sys2}
\left\{
\begin{array}{ll}
\partial_t^2u(x,t)-c_1\Delta u(x,t)=|v(x,t)|^p,\ &(x,t)\ \in\ \R^N\times[0,T),\\
\partial_t^2v(x,t)-c_2\Delta v(x,t)=|u(x,t)|^q,\ &(x,t)\ \in\ \R^N\times[0,T),
\end{array}
\right.
\end{equation}
where $c_1>0$ and $c_2>0$ are positive constants and denote the propagation speeds. In the case of $c_1=c_2$, the following statement is proved in \cite{AKH00,D99,GTZ,KO99,Kurokawa,KT,ISW19}:
\[
  \left\{
  \begin{array}{l}
  \text{if}\ \Gamma_{SS}(N,0,p,q)>0,\ \text{the solution exists globally in time for small data},\\
  \text{if}\ \Gamma_{SS}(N,0,p,q)<0,\ \text{the solution blows up in a finite time even for small data}.
\end{array}\right.
\]
Here the closed set $\left\{(p,q)\in (1,\infty)^2\ |\ \Gamma_{SS}(N,m,p,q)=0\right\}$ is called a critical $(p,q)$-curve. Also, the sharp lower and upper estimates of lifespan are shown as
$$
T(\varepsilon)\sim \left\{
\begin{array}{ll}
C \varepsilon^{-\Gamma_{SS}(N,0,p,q)^{-1}},&\quad \mbox{if}\ \  \Gamma_{SS}(N,0,p,q)>0,\\
\exp \left(C \varepsilon^{-\min\{p(pq-1),q(pq-1)\}}\right),&\quad \mbox{if}\ \  \Gamma_{SS}(N,0,p,q)=0,\ p\neq q,\\
\exp \left(C \varepsilon^{-p(p-1)}\right),&\quad \mbox{if}\ \  \Gamma_{SS}(N,0,p,q)=0,\ p=q,\\
\end{array}
\right.
$$
for any $\varepsilon\in (0,\varepsilon_0]$, where $\varepsilon_0$ and $C$ are positive constants independent of $\varepsilon$. See \cite{DGM97} for a hyperbolic system of Hamiltonian type. We emphasize that our main result extends this upper estimate to more general case $m\ge 0$ and also treat the one spatial dimension $N=1$. For a strongly coupled system of semilinear wave equations with unequal propagation speeds, Kubo and Ohta \cite{KO05,KO06} derived sharp conditions for the small data global existence and blowup, and obtained the estimate of the lifespan for the blowup solution in two and three space dimensions. And Zha and Zhou \cite{ZZ} considered the system of quasilinear wave equations with multiple propagation speeds in $\R^4$, and got the sharp lower bound of the lifespan.
%For other nonlinearities or the different propagation speeds $c_1\ne c_2$ to (\ref{sys2}), see \cite{KO05,KO06}.

\subsection{Main result}
\ \ In this subsection, we state our main small data blow-up results in the present paper. To do so, we introduce several terminologies and notation. First we give the definition of the weak solution to the Cauchy problem (\ref{sys}).
\begin{definition}[Weak solution]
Let $N\in \N$, $m_1,m_2\ge 0$, $f_1,f_2 \in H^1(\R^N)$, $g_1,g_2 \in L^2(\R^N)$, $\varepsilon>0$, $G_1,G_2:\mathbb{R}\rightarrow\mathbb{R}_{\ge 0}$ be functions satisfying the assumption (\ref{nonlinearity}) and $T>0$. A pair of functions
$$
(u,v)\in L^{q}_{loc}\left(\R^N\times (0,T)\right)\times L^{p}_{loc}\left(\R^N\times (0,T)\right)
$$
is called a (weak) solution to the problem \eqref{sys} on $[0,T)$, if the identities $(u,v)(0)=(\varepsilon f_1,\varepsilon f_2),\ \partial_t(u,v)(0)=(\varepsilon g_1,\varepsilon g_2)$ hold and for any test function $\Phi \in C_0^\infty(\R^N\times[0,T))$, the equations
\begin{equation}\label{weaksol}
\begin{aligned}
&\varepsilon\int_{\R^N}g_1\Phi(0)dx+\int_0^T\int_{\R^N}G_1(v(t))\Phi(x,t)dxdt\\
&\quad\quad\quad\le -\int_0^T\int_{\R^N} \partial_tu(t) \partial_t\Phi(x,t) dxdt+\int_0^T\int_{\R^N}t^{m_1} \nabla u(t)\cdot \nabla \Phi(x,t) dxdt,\\
&\varepsilon\int_{\R^N}g_2\Phi(0)dx+\int_0^T\int_{\R^N}G_2(u(t))\Phi(x,t)dxdt\\
&\quad\quad\quad\le-\int_0^T\int_{\R^N} \partial_tv(t)\partial_t\Phi(x,t) dxdt+\int_0^T\int_{\R^N}t^{m_2} \nabla v(t)\cdot \nabla \Phi(x,t)dxdt\\
\end{aligned}
\end{equation}
hold.
\end{definition}

\begin{remark}[Local existence of the weak solution]
 By similar argument as in \cite{D}, we could get a unique solution $(u,v)\in (C([0,T),H^{s_0}(\R^N)))^2$, with the property of finite propagation speed, to the Cauchy problem (\ref{sys}) with the initial data $f_1, f_2, g_1, g_2 \in H^{s_0+2}(\R^N))$ for $s_0>2$. In this case, we could obtain the local existence of the weak solution.
\end{remark}

Next we give the definition of lifespan $T(\varepsilon)$ of the solution to the problem (\ref{sys}).
\begin{definition}[Lifespan]
\label{Life}
We call the maximal existence time of the weak solution $(u,v)$ to the Cauchy problem (\ref{sys}) with the initial data $(u,v)(0)=(\varepsilon f_1,\varepsilon f_2),\ \partial_t(u,v)(0)=(\varepsilon g_1,\varepsilon g_2)$ to be lifespan, which is denoted by $T(\varepsilon)=T\left(\varepsilon, (f_1,g_1), (f_2,g_2)\right)$, that is
\[
   T(\varepsilon):=\sup\left\{T\in (0,\infty]\ |\ \text{there exists a weak solution}\ (u,v)\ \text{to (\ref{sys})}\ \text{on}\ [0,T)\right\}.
\]
\end{definition}
For $m\ge 0$, we set $\gamma=\gamma(m):=\frac {m+2}2\ge 1$. And for $N\in \N$ and $m\ge 0$, we introduce three rational functions $F_{SS}=F_{SS}(N,m,p,q)$, $\Gamma_{SS}=\Gamma_{SS}(N,m,p,q)$ and $F_{GG}=F_{GG}(N,m,p,q)$ of $p,q\ge 1$ given by
\begin{align*}
F_{SS}(N,m,p,q):&=-\frac{(m+2)(N-1)}{4}-\frac{m}{4}-\frac{m}{2q}+(pq-1)^{-1}\left(p+2+\frac 1q\right),\\
\Gamma_{SS}(N,m,p,q):&=\max\{F_{SS}(N,m,p,q),F_{SS}(N,m,q,p)\}\\
F_{GG}(N,m,p,q):&=\frac{2(p+1)}{pq-1}-(\gamma(m)N-1)\\
\Gamma_{GG}(N,m,p,q):&=\max\{F_{GG}(N,m,p,q),F_{GG}(N,m,q,p)\}
\end{align*}
For $m_1,m_2\ge 0$ with $m_1\ne m_2$, we also introduce a rational function $\Omega_{SS}=\Omega_{SS}(N,m_1,m_2,p,q)$ and $\Omega_{GG}=\Omega_{GG}(N,m_1,m_2,p,q)$
of $p,q\ge 1$ given by
\begin{align}
\label{criti}
\Omega_{SS}(N,m_1,m_2,p,q)&:=
\begin{cases}
     F_{SS}(N,m_1,p,q),&\ m_1>m_2,\\
     F_{SS}(N,m_2,q,p),&\ m_1<m_2,
\end{cases}\\
\Omega_{GG}(N,m_1,m_2,p,q):&=
\begin{cases}
     F_{GG}(N,m_1,p,q),&\ m_1\ge m_2,\\
     F_{GG}(N,m_2,p,q),&\ m_1<m_2,
\end{cases}
\end{align}
We introduce a functional $I:L^1(\R^N)\rightarrow \R$ given by
$$I[g]:=\int_{\R^N}g(x)dx.$$
For $m_1,m_2\ge 0$ and $T>0$, we define a function $\phi:(0,T)\rightarrow \R_{\ge 0}$ given by
\begin{equation}
\label{phi}
\phi(t)=\phi(m_1,m_2;t):=\max\left(\frac{1}{\gamma(m_1)} t^{\gamma(m_1)}, \frac{1}{\gamma(m_2)} t^{\gamma(m_2)}\right).
\end{equation}
%\begin{aligned}
%& \phi(t)=\phi(m_1,m_2;t)\\
%&:=\max \left\{\max_{0\leq s \leq t} \left(r_0+\frac{1}{\gamma(m_1)} (t-s)^{\gamma(m_1)}+\frac {1}{\gamma(m_2)} s^{\gamma(m_2)}\right), \max_{0\leq s \leq t} \left(r_0+\frac{1}{\gamma(m_2)} (t-s)^{\gamma(m_2)}+\frac {1}{\gamma(m_1)} s^{\gamma(m_1)}\right)\right\}\\
%&=r_0+\frac{1}{\gamma(m_1)} t^{\gamma(m_1)}.
%\end{aligned}
%\end{equation}

Now we state our first main result, which means small data blow-up and an upper bound of the lifespan to the problem (\ref{sys}) with different propagation speeds $m_1\neq m_2$ in the critical and subcritical cases $\Omega_{SS}(N,m_1,m_2,p,q)\ge 0$:
\begin{theorem}[Small data blow-up and upper estimate of the lifespan for $m_1\ne m_2$]
\label{Main}
Let $N\in \N$, $m_1,m_2\ge 0$ with $m_1\ne m_2$, $p,q>1$ satisfy $\Omega_{SS}(N,m_1,m_2,p,q)\geq 0$ and $G_1,G_2:\mathbb{R}\rightarrow\mathbb{R}_{\ge 0}$ be functions satisfying (\ref{nonlinearity}). Let $f_1,f_2 \in H^1(\R^N)$, $g_1,g_2 \in L^2(\R^N)$ satisfy the following estimates
$$I[g_1]>0,\quad \mbox{and}\quad I[g_2]>0.$$
Let $\varepsilon>0$. We assume that a pair of functions $(u,v)$ is a weak solution of the problem \eqref{sys} on $[0,T)$ with the finite propagation property, that is,
\begin{equation}
\label{FPS}
\supp(u,v)\subset \left\{(x,t)\in \R^N\times [0,T)\ \bigg|\ |x|\leq r_0+\phi(t)\right\},
\end{equation}
where $r_0:=\sup \left\{|x|\in \R_{\ge 0}\ |\ x\in \supp(f_1,f_2,g_1,g_2)\right\}>0$ and $\phi=\phi(t)$ is given by (\ref{phi}). Then the lifespan $T(\varepsilon)<\infty$. Moreover there exist positive constants $A=A(N,m_1,m_2,p,q,g_1,g_2)>0$ and $\varepsilon_0=\varepsilon_0(N,m_1,m_2,p,q,g_1,g_2)\in (0,1]$ independent of $\varepsilon$ such that for any $\varepsilon\in (0,\varepsilon_0]$, the estimate
\begin{equation}
\label{upper}
T(\varepsilon)\leq \left\{
\begin{array}{ll}
A \varepsilon^{-\Omega_{GG}(N,m_1,m_2,p,q)^{-1}},&\quad \mbox{if}\ \  \Omega_{GG}(N,m_1,m_2,p,q)>0,\\
A \varepsilon^{-\Omega_{SS}(N,m_1,m_2,p,q)^{-1}},&\quad \mbox{if}\ \  \Omega_{SS}(N,m_1,m_2,p,q)>0,\\
\exp (A \varepsilon^{-q(pq-1)}),&\quad \mbox{if}\ m_1>m_2,\ \Omega_{SS}(N,m_1,m_2,p,q)=0, \text{and}\ N\geq 2,\\
\exp (A \varepsilon^{-p(pq-1)}),&\quad \mbox{if}\ m_1<m_2,\ \Omega_{SS}(N,m_1,m_2,p,q)=0, \text{and}\  N\geq 2,\\
%\exp (C \varepsilon^{-p(p-1)}),&\quad \mbox{if}\ \  \Gamma_{SS}(N,m_1,m_2,p,q)=0,\ p=q\ \mbox{and}\ m_1=m_2,\\
\end{array}
\right.
\end{equation}
holds.
\end{theorem}

Next we study the problem (\ref{sys}) with the same propagation speed $m_1=m_2=:m>0$. In this case, we can prove a better upper estimate of the lifespan than Theorem \ref{Main} in the critical and subcritical cases $\Gamma_{SS}(N,m,p,q)\ge 0$:
\begin{theorem}[Small data blow-up and upper estimate of the lifespan for $m_1=m_2$]
\label{Main2}
Besides the same assumptions except for $m_1\ne m_2$ of Theorem \ref{Main}, we assume that $m_1=m_2=:m>0$. Then the lifespan is finite, i.e. $T(\varepsilon)<\infty$. Moreover there exist positive constants $B=B(N,m,p,q,g_1,g_2)>0$ and $\varepsilon_0=\varepsilon_0(N,m,p,q,g_1,g_2)\in (0,1]$ independent of $\varepsilon$ such that for any $\varepsilon\in (0,\varepsilon_0]$, the estimate
\begin{equation}
\label{upper1}
T(\varepsilon)\leq \left\{
\begin{array}{ll}
B \varepsilon^{-\Gamma_{SS}(N,m,p,q)^{-1}},&\quad \mbox{if}\ \  \Gamma_{SS}(N,m,p,q)>0,\\
\exp (B \varepsilon^{-\min\{p(pq-1),q(pq-1)\}}),&\quad \mbox{if}\ \  \Gamma_{SS}(N,m,p,q)=0,\ p\neq q, \mbox{and}\ N\geq 2,\\
\exp (B \varepsilon^{-p(p-1)}),&\quad \mbox{if}\ \  \Gamma_{SS}(N,m,p,q)=0,\ p=q, \mbox{and}\  N\geq 2,\\
\end{array}
\right.
\end{equation}
holds.
\end{theorem}
Our results (Theorems \ref{Main} and \ref{Main2}) extend the previous results\cite[Theorem 1.1]{KO99}, \cite[Theorem 1]{AKH00}, \cite[Theorem 1.1]{KTW12} for the classical case $m_1=m_2=0$ to the general case $m_1,m_2\ge 0$.

To understand the region of the pairs $(p,q)$ of the blow-up, that is, $\Omega_{SS}(N,m_1,m_2,p,q)\ge 0$ with $m_1\neq m_2$ or $\Gamma_{SS}(N,m,p,q)\ge 0$ with $m_1=m_2=m$ visually, we draw two pictures:
\begin{remark}[Region of $(p,q)$ of the blow-up with $m_1> m_2$]
\label{pq2}
For a point $(p,q)$, the identity $F_{SS}(N,m_1,p,q)=0$ holds on the blue curve. The inequality $F_{SS}(N,m_1,p,q)>0$ holds in the yellow area below the curve.
\begin{center}
\includegraphics[scale=0.6]{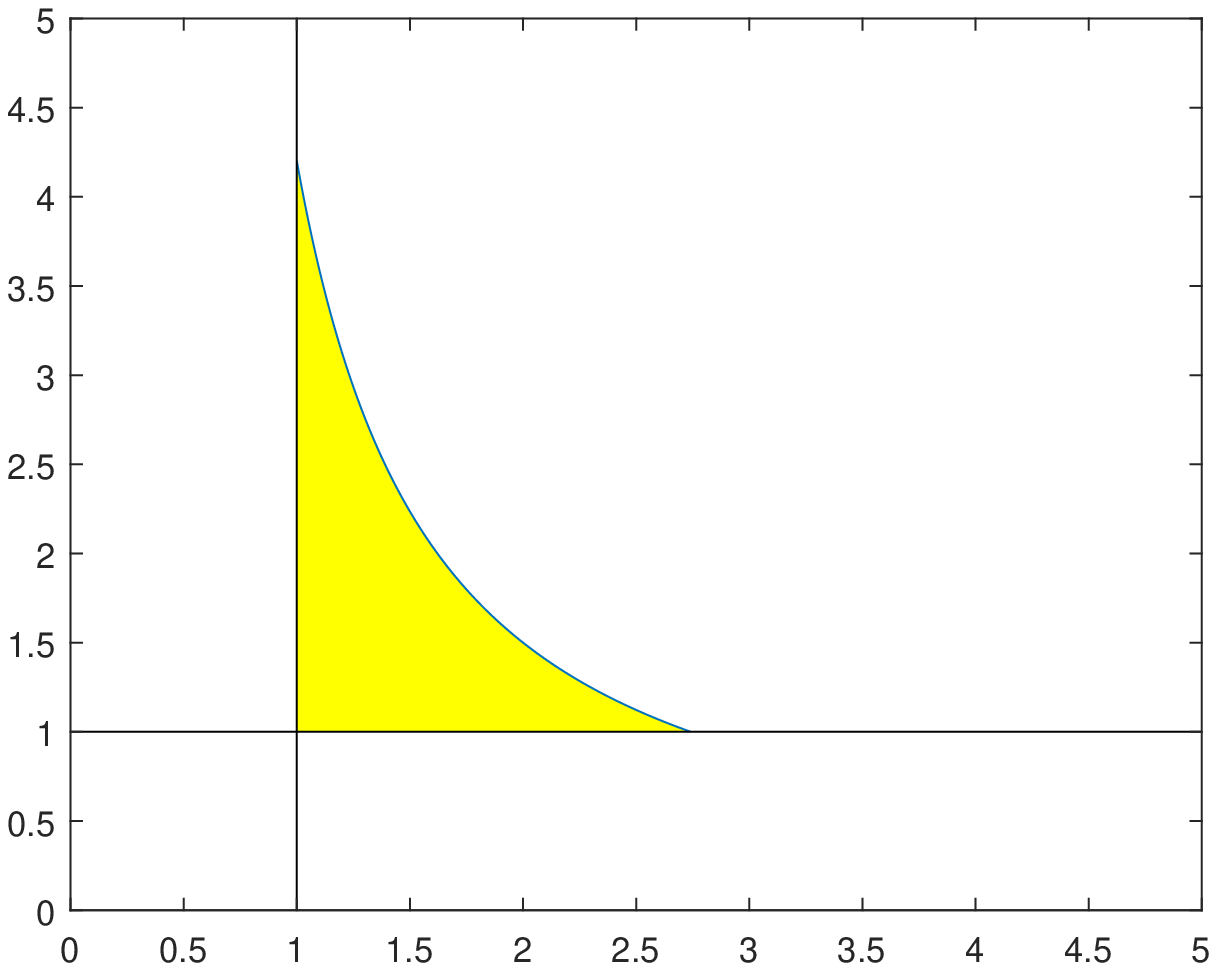}
%\caption{this is a figure demo}
%\label{fig:label}
\end{center}
\end{remark}

\begin{remark}[Region of $(p,q)$ of the blow-up with $m_1=m_2=m$]
\label{p-qcurve}
For a point $(p,q)$, the identity $F_{SS}(N,m,p,q)=0$ holds on the blue curve, while $F_{SS}(N,m,q,p)=0$ holds on the red curve. Thus $\Gamma_{SS}(N,m,p,q)>0$ holds in the yellow area below the two curves. Moreover $\Gamma_{SS}(N,m,p,q)=0$ and $p\neq q$ hold on the two curves. At the intersection point of the two curves, $\Gamma_{SS}(N,m,p,q)=0$ and $p=q$ hold.
\begin{center}
\includegraphics[scale=0.6]{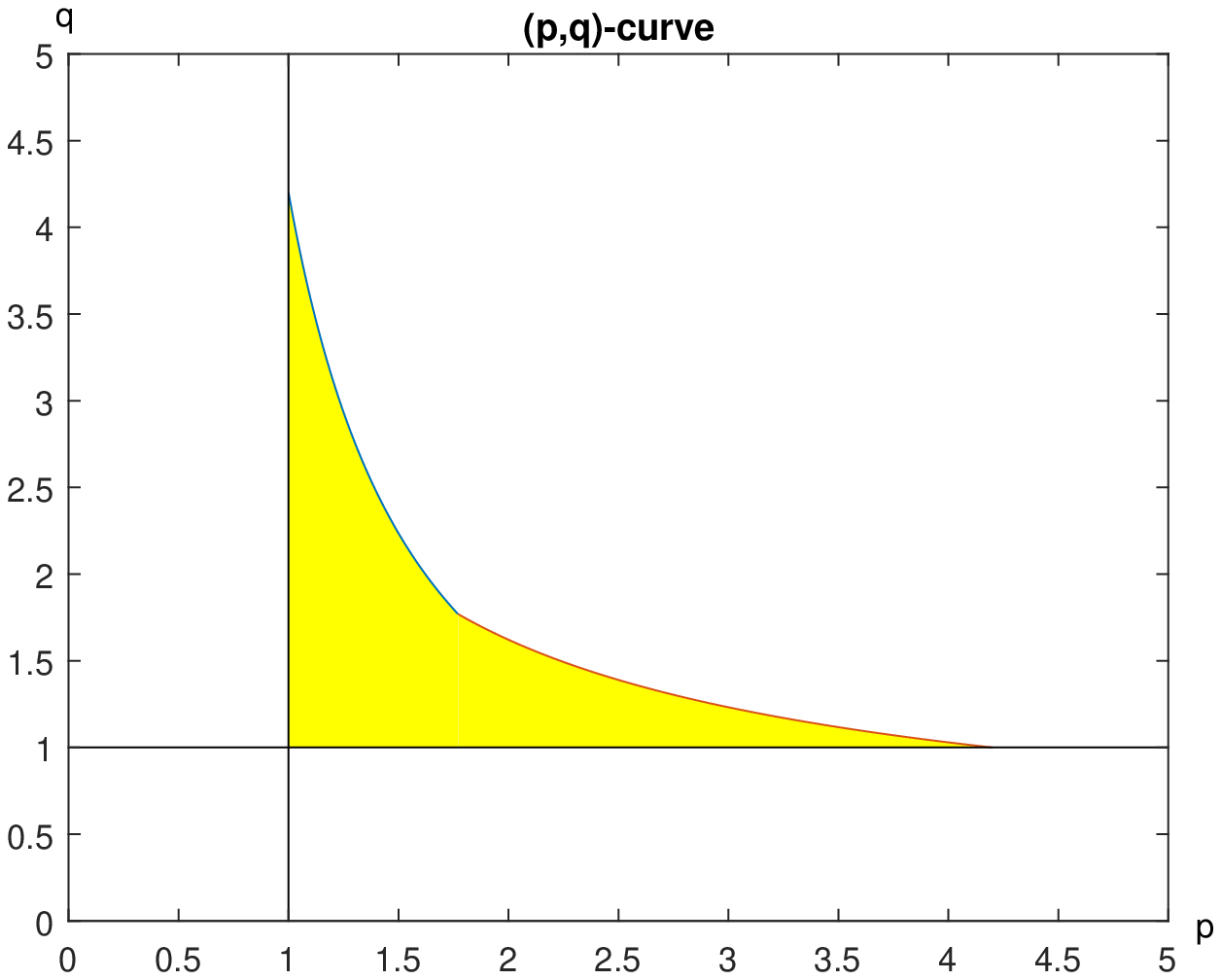}
%\caption{this is a figure demo}
%\label{fig:label}
\end{center}
\end{remark}

%delete Remark 1.3
%Next we explain our new contribution for the single equation (\ref{single}):
%\begin{remark}[Improvement of the previous results for (\ref{single})]
%\label{positi}
%By applying the technique for the proof of Theorem \ref{Main}, we can prove the corresponding small data blow-up results and the upper estimate of the lifespan to the Cauchy problem for the single equation (\ref{single}) similar to those obtained in \cite[Theorem 1.1]{He-Witt-Yin17}, \cite[Theorem 1.1]{He-Witt-Yin17-2} and \cite[Theorem 1.2 and 1.3]{LTpre}. However, we do not need to assume that the initial data $f$ and $g$ are positive in the pointwise sense assumed in the previous results and we only need the assumption $I[g]=\int_{\R^N} g(x) dx>0$ (see also \cite{ISW19, SW20} for the classical case $m=0$).
%\end{remark}

\subsection{Strategy of the proof}\label{strate}
\ \ We explain the strategy of the proof of Theorem \ref{Main}.
The proof is based on the framework of the argument on the test function in \cite{ISW19}, which is effective to prove small data blow-up results and the sharp upper estimates of lifespan to the classical several semilinear wave equations ($m_1=m_2=0$) in the critical and subcritical cases. To apply the method to our general case ($m_1, m_2\ge 0$), we require to construct two families of special solutions $\{w_{\lambda,m}(x,t)\}_{\lambda>0}$ (Definition \ref{specialsolsub}) and $\{W_{\beta,m}(x,t)\}_{\beta>0}$ (Definition \ref{specialsol}) to the free generalized Tricomi equation
\begin{equation}
\label{freeeq}
\partial_t^2w(x,t)-t^m\Delta w(x,t)=0,\ \ \ (x,t)\in \R^N\times \R_{\ge 0},
\end{equation}
where $m\ge 0$, and prove several estimates for the solutions $w_{\lambda,m}$ and $W_{\beta,m}$ (Lemma \ref{propy} and Proposition \ref{propertest}). The family $\{w_{\lambda,m}\}_{\lambda>0}$ is employed to treat both subcritical and critical cases $\Omega_{SS}\ge 0$ with $m\ne m_2$ and $\Gamma_{SS}\ge 0$ with $m_1=m_2$ and especially to derive a lower estimate (\ref{equ3}) of the solution to the semilinear problem (\ref{sys}).
The other family $\{W_{\beta,m}\}_{\beta>0}$ is used to deal with the only critical cases $\Omega_{SS}= 0$ with $m\ne m_2$ and $\Gamma_{SS}=0$ with $m_1=m_2$ and in particular to prove an upper estimate (\ref{3-21}) of the solution to (\ref{sys}). The special solution $w_{\lambda,m}$ is constructed of the form $w_{\lambda,m}(x,t)=y_{\lambda,m}(t)\varphi_{\lambda}(x)$, where $\varphi_{\lambda}$ is defined by (\ref{test2}) and is an eigenfunction of the Laplace operator $\Delta$ with respect to the eigenvalue $\lambda^2$. We define the function $y_{\lambda,m}$ as a special solution to the second order ordinary differential equation
$$y"-\lambda^2 t^m y = 0,\ \ \ t\ge 0,$$
and especially $y_{\lambda,m}$ is expressed in terms of the modified Bessel function of the second kind (\ref{mBessel}). Also, the other special solution $W_{\beta,m}(x,t)$ is defined in Definition \ref{specialsol}, and especially, it is parametrized by $\beta>0$ and integrated by the parameter $\lambda$ on $(0,1)$. Several estimates (Lemma \ref{propy} and Proposition \ref{propertest}) of the functions $w_{\lambda,m}$ and $W_{\beta,m}$ are proved by using properties of the function $\varphi_{\lambda}$ and the modified Bessel function.
The proof of the small data blow-up and the upper estimate of the lifespan to (\ref{sys}) is completed by combining the properties of the above test functions and Lemma 3.10 in \cite{ISW19} (see also \cite[Proposition 7.1]{ISW19}).

%\subsection{Organization of the present paper}
%The rest of this paper is organized as follows.
%In section \ref{Preli}, we construct two families of special solutions $\{w_{\lambda,m}(x,t)\}_{\lambda>0}$ and $\{W_{\beta,m}(x,t)\}_{\beta>0}$ to the free generalized Tricomi equation with $m\ge 0$, and prove their several properties. In section \ref{Proofmain}, we give proofs of Theorem \ref{Main} and Theorem \ref{Main2} by applying the test function method developed in \cite{ISW19} with appropriate test functions including the special solutions constructed in section \ref{Preli}.

%% file: Preliminaries.tex
\section{Construction of two families of special solutions and their properties}
\label{Preli}
\ \ \ In this section, we introduce two families of special solutions $\{w_{\lambda,m}(x,t)\}_{\lambda>0}$ and $\{W_{\beta,m}(x,t)\}_{\beta>0}$ given by  (\ref{testsub}) and (\ref{testcri}) respectively, to the free generalized Tricomi equation (\ref{freeeq}) and prove their several properties (Lemma \ref{propy}, Corollary \ref{estphi} and Proposition \ref{propertest}), which are used to prove Theorems \ref{Main} and \ref{Main2}.

For $m\ge 0$ and $\lambda>0$, we consider the following second order ordinary differential equation
\begin{equation}
\label{ODE}
y"-\lambda^2 t^m y = 0,\ \ \ t\ge 0,
\end{equation}
where $y=y(t)$ is an unknown function. As shown in \cite[Section 2.1.2]{ODE}, one of the solutions to (\ref{ODE}) can be expressed as
\begin{equation}
\label{ODEsol}
y(t)=y_{\lambda}(t)=y_{\lambda,m}(t)=C_0 (\lambda^{\frac{1}{\gamma}}t)^{\frac{1}{2}}K_{\frac 1{2\gamma}}\left(\gamma^{-1}(\lambda^{\frac{1}{\gamma}} t)^\gamma\right),\ \ \ t\ge 0,
\end{equation}
where $\gamma=\gamma(m)=\frac {m+2}2$ and $C_0\in \R$ is an arbitrary number. Here for real $\nu>0$, $K_{\nu}:\R_{\ge 0}\rightarrow \R_{\ge 0}$ is the modified Bessel function of the second kind given by
\begin{equation}
\label{mBessel}
K_\nu(t):=\int_0^\infty e^{-t\cosh z}\cosh (\nu z)dz.
\end{equation}
From \cite[Page24]{E53}, it is seen that the asymptotic estimate
$K_{\nu}(t)=\sqrt{\frac{\pi}{2t}}e^{-t}\left(1+O\left(t^{-1}\right)\right)$ holds as $t\rightarrow\infty$ if $\nu>\frac{1}{2}$, which implies that there exist $C>1$ and $t_0=t_0(\nu)\ge 1$ dependent only on $\nu$ such that for any $t\in [t_0,\infty)$, the estimates \begin{equation}
\label{properK}
C^{-1}t^{-1/2}e^{-t}\le K_\nu(t)\le C t^{-1/2}e^{-t}
\end{equation}
hold. Here and in what follows, all positive constants, except for some constants with special values, are donated by $C$, which may vary from line to line.
Noting that the identity $\lim_{t\rightarrow+0}t^{\nu}K_{\nu}(t)=2^{\nu-1}\Gamma(\nu)$ holds for $\nu>0$, where $\Gamma:\R_{\ge 0}\rightarrow \R_{> 0}$ is the Gamma function given by $\Gamma(s):=\int_0^{\infty}z^{s-1}e^{-z}dz$, we choose the constant $C_0$ as $C_0=C_0(\gamma):=\gamma^{-\frac 1{2\gamma}}2^{1-\frac 1{2\gamma}}\Gamma(\frac 1{2\gamma})^{-1}$, so that $y_{\lambda}(t)$ given by (\ref{ODEsol}) satisfies
\begin{equation}
\label{propy0}
y_{\lambda}(+0):=\lim_{t\rightarrow+0}y_{\lambda}(t)=1\ \ \ \text{and}\ \ \  y_{\lambda}(\infty):=\lim_{t\rightarrow\infty}y_{\lambda}(t)=0.
\end{equation}
From a property $\frac{d}{dt}\left(t^{\nu}K_{\nu}(t)\right)=-t^{\nu}K_{\nu-1}(t)$, it can be seen that
\begin{equation}
\label{y'(0)}
y'_{1}(+0):=\lim_{t\rightarrow+0}y'_{1}(t)=:-C_0'<0,
\end{equation}
where $C_0'$ depends only on $m$.
We state several estimates of the solution $y_{\lambda}(t)$ to (\ref{ODE}).
\begin{lemma}[Properties of $y_{\lambda}$]
\label{propy}
Let $m\ge 0$, $\lambda>0$ and $y_{\lambda}=y_{\lambda,m}:\R_{\ge 0}\rightarrow \R_{>0}$ be defined by (\ref{ODEsol}). Then the following statements hold:
\begin{enumerate}
\item The functions $y_{\lambda}$ and $-y_{\lambda}'$ are decreasing on $[0,\infty)$ and the identities $\lim_{t\rightarrow \infty}y'_{\lambda}(t)=\lim_{t\rightarrow \infty}y_{\lambda}(t)=0$ hold.
    \item There exist $T_0=T_0(m)\ge 1$ and $C=C(m)>1$ such that for any $t\ge 0$ with $\lambda t^{\gamma}\ge T_0$, the estimates hold:
\begin{equation}
\label{esty}
C^{-1}(\lambda^{\frac{1}{\gamma}}t)^{-\frac{m}{4}}\exp\left(-\gamma^{-1}(\lambda^{\frac{1}{\gamma}}t)^{\gamma}\right)\le y_{\lambda}(t)\le C(\lambda^{\frac{1}{\gamma}}t)^{-\frac{m}{4}}\exp\left(-\gamma^{-1}(\lambda^{\frac{1}{\gamma}}t)^{\gamma}\right).
\end{equation}
\item There exists a constant $C=C(m)>1$ such that for any $t\ge 0$ with $\lambda^{\frac{1}{\gamma}}t\ge 1$, the estimates hold:
\begin{align}
\label{estdy}
    C^{-1}t^{\frac{m}{2}}\lambda y_{\lambda}(t)\le |\partial_ty_{\lambda}(t)|\le Ct^{\frac{m}{2}}\lambda y_{\lambda}(t).
\end{align}
Moreover there exist $T_1=T_1(m)>1$ and $C=C(m)>1$ such that for any $t\ge 0$ with $\lambda^{\frac{1}{\gamma}}t\ge T_1$, the estimates hold:
\begin{equation}
\label{estdy2}
    C^{-1}\lambda^{\frac{1}{2}+\frac{1}{m+2}}t^{\frac{m}{4}}\exp\left(-\gamma^{-1}(\lambda^{\frac{1}{\gamma}}t)^{\gamma}\right)\le |\partial_ty_{\lambda}(t)|\le C\lambda^{\frac{1}{2}+\frac{1}{m+2}}t^{\frac{m}{4}}\exp\left(-\gamma^{-1}(\lambda^{\frac{1}{\gamma}}t)^{\gamma}\right).
\end{equation}
\end{enumerate}
\end{lemma}

\begin{proof}[Proof of Lemma \ref{propy}]
1. From a simple computation, we see that the identity $y_{\lambda}(t)=y_1(\lambda^{\frac{1}{\gamma}}t)$ holds for any $t\in \R_{\ge 0}$. By combining this fact and \cite[Lemma 2.1]{HL}, the conclusion hold (see also \cite[Lemma 2.2]{He-Witt-Yin17}).\\
2. Set $T_0=T_0(m):=\max\left(1,\left(\gamma t_0\left(\frac{1}{2\gamma}\right)\right)^{\frac{1}{\gamma}}\right)$ and let $t\ge 0$ with $\lambda t^{\gamma}\ge T_0$. Then the inequality $\gamma^{-1}\lambda t^{\gamma}\ge t_0$ holds, which enables us to apply the function $y_{\lambda}$ to the asymptotic estimate (\ref{properK}), to obtain the estimates (\ref{esty}) with $C=C(m):=C_1\gamma^{-\frac{1}{2}}\max\big(C_0,C_0^{-1}\big)>1$.\\
3. From \cite[Lemma 2.1]{HL}, we see that there exists a constant $C=C(m)>1$ such that for any $t\in [1,\infty)$, the estimates
\begin{equation}
\label{propy1}
C^{-1} t^{\frac{m}{2}} y_1(t)\le |y_1'(t)| \le C t^{\frac{m}{2}} y_1(t)
\end{equation}
hold. By these properties, (\ref{estdy}) is proved. Set $T_1=T_1(m):=\max\big(T_0,1\big)$. By combining the estimates (\ref{esty}) and (\ref{estdy}), (\ref{estdy2}) is proved, which completes the proof of the lemma.
\end{proof}

Next we introduce a positive function $\varphi:\R^N\rightarrow \R_{> 0}$ defined by
\begin{equation}
\label{test}
\varphi(x):=
\left\{
\begin{array}{cc}
\displaystyle{\frac{1}{\left|\mathbb{S}^{N-1}\right|}\int_{\mathbb{S}^{N-1}}e^{x\cdot \omega}d\omega}, \ \ &\mbox{if}\ \ N\geq 2,\\
\displaystyle{\frac{1}{2}(e^x+e^{-x})}, \ \ &\mbox{if}\ \ N=1,
\end{array}
\right.
\end{equation}
where $\mathbb{S}^{N-1}:=\big\{x\in \mathbb{R}^N\ |\ |x|=1\big\}$ denotes the $(N-1)$-dimensional unit sphere in $\R^N$ and $\left|\mathbb{S}^{N-1}\right|$ denotes the Lebesgue measure of $\mathbb{S}^{N-1}$. In the case of $N\ge 2$, the function $\varphi$, which is a spherical average of $e^{x\cdot\omega}$ and was used to prove a small data blow-up result \cite[Theorem 1.1]{YZ06} to the semilinear classical ($m=0$) wave equation (\ref{classi}) in the critical case $\rho=\rho_c(N,0)$.

For $\lambda>0$, we define a positive function $\varphi_{\lambda}:\R^N\rightarrow\R_{> 0}$ given by
\begin{equation}
\label{test2}
\varphi_\lambda(x):=\varphi(\lambda x).
\end{equation}
Next we state several properties of the function $\varphi_{\lambda}(x)$.

\begin{lemma}[Properties of $\varphi_{\lambda}$]
\label{Property}
Let $N\in \N$, $\lambda>0$ and $\varphi_{\lambda}:\R^N\rightarrow\R_{> 0}$ be defined by (\ref{test2}). Then the following statements hold:
\begin{enumerate}
\item Let $N\geq 2$. Then for any $x\in \R^N$, the identity holds:
\begin{equation}
\label{2-2-1}
\varphi_{\lambda}(x)=\int_{-1}^1(1-\theta^2)^{\frac {N-3}2}e^{\theta \lambda |x|}d\theta .
\end{equation}
\item For any $x\in \R^N$, the equation $\Delta \varphi_{\lambda}(x)=\lambda^2\varphi(x)$, the inequalities $1\le \varphi_{\lambda}(x)\le e^{\lambda|x|}$ and the identity $\lim_{\lambda\rightarrow +0}\varphi_{\lambda}(x)=1$ hold.
\item The asymptotic estimate $\varphi_{\lambda}(x)\sim |\lambda x|^{-\frac {N-1}2}e^{\lambda|x|}$ holds as $\lambda|x|\rightarrow \infty$.
\end{enumerate}
From these estimates, there exists a constant $C\ge 1$ independent of $\lambda$ such that for any $x\in \R^N$, the estimates hold:
\begin{equation}
\label{2-11-2}
    C^{-1}(1+\lambda|x|)^{-\frac{N-1}{2}}e^{\lambda|x|}\le \varphi_{\lambda}(x)\le C(1+\lambda|x|)^{-\frac{N-1}{2}}e^{\lambda|x|}.
\end{equation}
\end{lemma}
For the proof of this lemma, see  \cite[P.8]{John} and  \cite[P.364]{YZ06}. We also estimate the Lebesgue norm of $\varphi_{\lambda}$.
\begin{corollary}
\label{estphi}
Let $\theta\in (1,\infty)$, $\lambda\in (0,1)$ and $\varphi_{\lambda}:\R^N\rightarrow\R_{>0}$ be defined by (\ref{test2}). Then there exist constants $C=C(N,\theta)>1$ independent of $\lambda$ and $R_0=R_0(N,\theta,\lambda)>0$ such that for any $\mathcal{R}\ge \mathcal{R}_0$, the estimate
\begin{equation}
    \int_{|x|<\mathcal{R}}\{\varphi_{\lambda}(x)\}^{\theta}dx\le C\lambda^{-\frac{N-1}{2}\theta-1}(1+\mathcal{R})^{N-1-\frac{N-1}{2}\theta}e^{\lambda\theta \mathcal{R}}
\end{equation}
holds.
\end{corollary}

\begin{proof}[Proof of Corollary \ref{estphi}]
We note that there exists $\mathcal{R}_0=\mathcal{R}_0(N,\theta,\lambda)>0$ such that for any $\mathcal{R}\ge \mathcal{R}_0$, the estimate $(1+R)^{-(N-1)(1-\frac{\theta}{2})}e^{\frac{3}{4}\lambda\theta\mathcal{R}}\le e^{\lambda\theta\mathcal{R}}$ holds. Let $\mathcal{R}\ge\mathcal{R}_0$. By this fact, the right inequality of (\ref{2-11-2}), the assumption $\lambda\in (0,1)$ and changing variables, the estimates
\begin{align*}
\label{2-12}
    \int_{|x|<\mathcal{R}}\{\varphi_{\lambda}(x)\}^{\theta}dx&\le C^{\theta}\int_{|x|\le \mathcal{R}}(1+\lambda|x|)^{-\frac{N-1}{2}\theta}e^{\lambda\theta|x|}dx\\
    &\le  C^{\theta}\lambda^{-N+1}\left|\mathbb{S}^{N-1}\right|\int_0^{\mathcal{R}}(1+\lambda r)^{-\frac{N-1}{2}\theta+N-1}e^{\lambda\theta r}dr\\
    &=C^{\theta}\lambda^{-N+1}\left|\mathbb{S}^{N-1}\right|\left(\int_0^{\frac {\mathcal{R}}2}(1+\lambda r)^{-\frac{N-1}{2}\theta+N-1}e^{\lambda\theta r}dr+\int_{\frac {\mathcal{R}}2}^{\mathcal{R}}(1+\lambda r)^{-\frac{N-1}{2}\theta+N-1}e^{\lambda\theta r}dr\right)\\
    &\le C\lambda^{-\frac{N-1}{2}\theta}e^{\frac 34\lambda\theta \mathcal{R}}+C_6 \lambda^{-\frac{N-1}{2}\theta-1}(1+\mathcal{R})^{N-1-\frac{N-1}{2}\theta}e^{\lambda\theta \mathcal{R}}\\
    &\le C\lambda^{-\frac{N-1}{2}\theta-1}(1+\mathcal{R})^{N-1-\frac{N-1}{2}\theta}e^{\lambda\theta \mathcal{R}}
\end{align*}
hold, which completes the proof of the corollary.
%\begin{align*}
%    \mathcal{I}(\theta)=\int_0^{\frac{ r}{2}}(1+\rho)^{-\frac{N-1}{2}\theta+N-1}e^{\lambda\theta\rho}d\rho+\int_{\frac{ r}{2}}^{ r}(1+\rho)^{-\frac{N-1}{2}\theta+N-1}e^{\lambda\theta\rho}d\rho\\
%    \le C
%    +C(1+r)^{
%    hyn }
%\end{align*}
%Since $\varphi(x)\leq C_N (1+|x|)^{-\frac %{n-1}2}e^{|x|}$,
%$$
%\begin{aligned}
%\int_{B(r_0+\frac 1\gamma t^\gamma)}\left(\varphi_\lambda(x)\right)^{q'} dx&\leq C_N \int_{B(r_0+\frac 1\gamma t^\gamma)}\left((1+|\lambda x|)^{-\frac {n-1}2}e^{\lambda|x|}\right)^{q'} dx\\
%&\leq C \lambda^{-\frac {N-1}2 q'}\int_0^{r_0+\frac 1\gamma t^\gamma} (1+r)^{-\frac {N-1}2 q'+N-1} e^{q'\lambda r} dr\\
%&\leq C\lambda^{-\frac {N-1}2 q'}\int_0^{\frac {r_0+\frac 1\gamma t^\gamma}2} (1+r)^{-\frac {N-1}2 q'+N-1} e^{q'\lambda r} dr+C\lambda^{-\frac {N-1}2 q'}\int_{\frac {r_0+\frac 1\gamma t^\gamma}2}^{r_0+\frac 1\gamma t^\gamma} (1+r)^{-\frac {N-1}2 q'+N-1} e^{q'\lambda r} dr\\
%&\leq C\lambda^{-\frac {N-1}2 q'-1} e^{\frac 34 \lambda(r_0+\frac 1\gamma t^\gamma)q'}+C\left(r_0+\frac 1\gamma t^\gamma \right)^{N-1-\frac {N-1}2 q'}e^{\lambda(r_0+\frac 1\gamma t^\gamma)q'}\\
%&\leq C \lambda^{-\frac {N-1}2 %q'-1}\left(r_0+\frac 1\gamma t^\gamma %\right)^{N-1-\frac {N-1}2 %q'}e^{\lambda(r_0+\frac 1\gamma %t^\gamma)q'}.
%\end{aligned}
%$$
\end{proof}

Next we introduce a family of special solutions $\{w_{\lambda,m}\}_{\lambda>0}$ to the free equation (\ref{freeeq}).
\begin{definition}[A family of special solutions I]
\label{specialsolsub}
Let $m\ge 0$ and $\lambda>0$. Let $y_{\lambda}=y_{\lambda,m}:\R_{\ge 0}\rightarrow\R_{> 0}$ and $\varphi_{\lambda}:\R^N\rightarrow\R_{>0}$ be defined by (\ref{ODEsol}) and (\ref{test2}), respectively.
We define a function $w_{\lambda}=w_{\lambda,m}:\R^N\times\R_{\ge 0}\rightarrow\R_{>0}$ given by
\begin{equation}
\label{testsub}
w_{\lambda}(x,t)=w_{\lambda,m}(x,t):=y_{\lambda,m}(t)\varphi_{\lambda}(x).
\end{equation}
\end{definition}
We remark that for any $\lambda>0$, the equation (\ref{ODE}) and Lemma \ref{Property} imply that the function $w_{\lambda}$ satisfies
\begin{equation}
\label{freesol1}
    \partial_t^2w_{\lambda}(x,t)-t^m\Delta w_{\lambda}(x,t)=0,\ \ \ (x,t)\in \R^N\times\R_{\ge 0}.
\end{equation}
Several properties of the family $\{w_{\lambda,m}\}_{\lambda>0}$ play an important role to derive the estimate (\ref{equ3}) in the proof of Theorems \ref{Main} and \ref{Main2}.

Next we introduce another family of special solutions $\{W_{\beta,m}(x,t)\}_{\beta>0}$ to the free generalized Tricomi equation (\ref{freeeq}).

\begin{definition}[A family of special solutions II]
\label{specialsol}
Let $m\ge 0$ and $\{w_{\lambda,m}(x,t)\}_{\lambda>0}$ be the family of the functions given by (\ref{testsub}). For $\beta>0$, we define a positive function $W_{\beta}=W_{\beta,m}:\R^N\times\R_{\ge 0}\rightarrow\R_{>0}$ given by
\begin{equation}
\label{testcri}
W_\beta(x,t)=W_{\beta,m}(x,t):=\int_{0}^{1}w_{\lambda,m}(x,t)\lambda^{\beta-1}d\lambda.
\end{equation}
\end{definition}
From Lemma \ref{propy} for the function $y_{\lambda}$, Lemma \ref{Property} for the function $\varphi_{\lambda}$ and the condition $\beta>0$, we see that the function $W_\beta$ is well defined, i.e. $|W_{\beta}(x,t)|<\infty$ for any $(x,t)\in \R^N\times \R_{\ge 0}$. We remark that since $w_{\lambda}$ satisfies (\ref{freesol1}) on $\R^N\times\R_{\ge 0}$, the function $W_{\beta}$ satisfies
\begin{equation}
\label{solW}
\partial_t^2 W_\beta(x,t)-t^m \Delta W_\beta(x,t) =0,\ \ \ (x,t)\in \R^N\times\R_{\ge 0}.
\end{equation}
Moreover we state several estimates for the solution $W_{\beta,m}$ in the following proposition.
\begin{proposition}[Properties of the function $W_{\beta,m}$]
\label{propertest}
Let $m\ge 0$, $\beta>0$ and $W_{\beta,m}:\R^N\times \R_{\ge 0}\rightarrow \R_{>0}$ be defined by (\ref{testcri}). Then the following statements hold:
\begin{enumerate}
\item For any $(x,t)\in \R^N\times(1,\infty)$, the estimate
\begin{equation}\label{2-19}
    |\partial_t W_{\beta,m}(x,t)|\le C\exp(t^{-\gamma}|x|)t^{-\gamma(\beta+\frac{1}{\gamma})}+Ct^{\frac{m}{2}}W_{\beta+1,m}(x,t)
\end{equation}
holds for some positive constant $C=C(m)>1$.
\item Let $T_0\ge 1$ be given by (\ref{esty}). Then there exists a positive constant $C=C(m,\beta)>0$ dependent only on $m$ and $\beta$ such that for any $(x,t)\in \R^N\times\big((2T_0)^{\frac{1}{\gamma}},\infty\big)$, the estimate
\begin{equation}\label{Pro2.4.2}
    W_{\beta,m}(x,t) \geq C t^{-\frac {m+2}2 \beta}
\end{equation}
holds.
\item Let $r_0>0$ and $T_0\ge 1$ be given by (\ref{esty}). We further assume that $\beta > \frac{N}{2}-\frac{1}{m+2}$. Then there exists a positive constant $C=C(N,m,\beta,r_0)>0$ such that for any $(x,t)\in \R^N\times\big(T_0^{\frac{1}{\gamma}},\infty\big)$ with $|x|<r_0+\frac{1}{\gamma}t^{\gamma}$, the estimate holds:
\begin{equation}
\label{Pro2.4.3}
W_{\beta,m}(x,t) \leq C t^{-\frac m4-\frac{(N-1)(m+2)}{4}}\left(r_0+\frac 1\gamma t^\gamma-|x|\right)^{-\beta-\frac{1}{m+2}+\frac{N}{2}}.
\end{equation}
\item Let $N\geq 2$ and $\frac 12 - \frac 1{m+2}<\beta<\frac {N}2 - \frac 1{m+2}$. Then there exists a positive constant $C=C(N,m,\beta,r_0)>0$ such that for any $(x,t)\in \R^N\times\big(T_0^{\frac{1}{\gamma}},\infty\big)$ with $|x|<r_0+\frac{1}{\gamma}t^{\gamma}$,
\begin{equation}\label{Pro2.4.4}
W_{\beta,m}(x,t)\leq C t^{-\frac m4}\left(r_0+\frac {1}{\gamma}t^{\gamma}+|x|\right)^{-\beta+\frac 12-\frac 1{m+2}}.
\end{equation}
%$W_\beta \leq C t^{-\frac m4}(1+\frac 1\gamma)^{-\frac {N-1}2}(1+\frac 1\gamma t^\gamma-|x|)^{-\beta+\frac {N-1}2}$ for $|x|<r_0+\frac 1\gamma t^\gamma$, $t<R$ and $\beta>\frac {N-1}2$.
\end{enumerate}
\end{proposition}
These estimates of the function $W_{\beta,m}(x,t)$ are utilized to treat the critical cases $\Omega_{SS}=0$ and $\Gamma_{SS}=0$ in the proof of Theorems \ref{Main} and \ref{Main2} respectively (see subsection \ref{Cri} for more precise).

\begin{proof}[Proof of Proposition \ref{propertest}]
%1. Noting that $y(t)=C_0 \sqrt{t}K_{\frac 1{2\gamma}}\left(\frac{\lambda}{\gamma}t^\gamma\right)$ satisfies $y"-\lambda^2 t^m y = 0$ and $\varphi_\lambda(x)$ satisfies $\Delta \varphi_\lambda(x)=\lambda^2\varphi_\lambda(x)$, it follows that $C_0 \sqrt{t}K_{\frac 1{2\gamma}}\left(\frac{\lambda}{\gamma}t^\gamma\right) \varphi_\lambda(x)$ satisfies the following partial differential equation:
%$$
%\partial_t^2 w(x,t)-t^m \Delta w(x,t) =0,
%$$
%which implies the equation $\partial_t^2 w_\beta(x,t)-t^m \Delta w_\beta(x,t) =0$ holds.

1. %Let $h(t):=\sqrt{t}K_{\frac 1{2\gamma}}\left(\frac{1}{\gamma}t^\gamma\right)$.
Let $(x,t)\in \R^N\times (1,\infty)$. Then we see that $t^{-\gamma}\in (0,1)$. By a direct computation, the identities
\begin{equation}
\label{dW}
    \partial_tW_{\beta}(x,t)=\int_0^{t^{-\gamma}}(\partial_ty_{\lambda}(t))\varphi_{\lambda}(x)\lambda^{\beta-1}d\lambda
    +\int_{t^{-\gamma}}^1(\partial_ty_{\lambda}(t))\varphi_{\lambda}(x)\lambda^{\beta-1}d\lambda=:\mathcal{I}_{\beta}^1(x,t)+\mathcal{I}_{\beta}^2(x,t)
\end{equation}
hold. By the identities $\partial_ty_{\lambda}(t)=\lambda^{\frac{1}{\gamma}}y'_1(\lambda^{\frac{1}{\gamma}}t)$ and (\ref{y'(0)}), Lemma \ref{Property} and $\beta>0$, the estimates
\begin{align}
    \left|\mathcal{I}_{\beta}^1(x,t)\right|
    \le C_0'\int_0^{t^{-\gamma}}e^{\lambda|x|}\lambda^{\beta-1+\frac{1}{\gamma}}d\lambda\le C_0'\exp\left(t^{-\gamma}|x|\right)\int_0^{t^{-\gamma}}\lambda^{\beta-1+\frac{1}{\gamma}}d\lambda\le C_0'\exp(t^{-\gamma}|x|)t^{-\gamma\left(\beta+\frac{1}{\gamma}\right)}
\end{align}
hold. By the right estimate of (\ref{estdy}), the inequalities
\begin{align}
\label{estI2}
    \left|\mathcal{I}_{\beta}^2(x,t)\right|\le \int_{t^{-\gamma}}^1|\partial_ty_{\lambda}(t)|\varphi_{\lambda}(x)\lambda^{\beta-1}d\lambda
    \le C t^{\frac{m}{2}}\int_{t^{-\gamma}}^1y_{\lambda}(t)\varphi_{\lambda}(x)\lambda^{\beta}d\lambda\le C t^{\frac{m}{2}}W_{\beta+1}(x,t)
\end{align}
hold.
 By combining (\ref{dW})-(\ref{estI2}), we obtain the conclusion.

2. Let $(x,t)\in\R^N\times\big((2T_0)^{\frac{1}{\gamma}},\infty\big)$. Then we see that the inequalities $0<T_0t^{-\gamma}<2T_0t^{-\gamma}<1$ hold. By the left estimate of (\ref{esty}), Lemma \ref{Property} and the assumption $\beta>0$, the estimates

\begin{align*}
    W_{\beta}(x,t)&\ge \int_{T_0t^{-\gamma}}^{2T_0t^{-\gamma}}y_{\lambda}(t)\varphi_{\lambda}(x)\lambda^{\beta-1}d\lambda
    \ge C t^{-\frac{m}{4}}\int_{T_0t^{-\gamma}}^{2T_0t^{-\gamma}}
    \exp\left(-\gamma^{-1}\lambda t^{\gamma}\right)\lambda^{\beta-\frac{3}{2}+\frac{1}{m+2}}d\lambda\\
    &\ge C\exp(-2\gamma^{-1}T_0)t^{-\frac{m}{4}}\int_{T_0t^{-\gamma}}^{2T_0t^{-\gamma}}
    \lambda^{\beta-\frac{3}{2}+\frac{1}{m+2}}d\lambda
    =C t^{-\frac {m+2}2 \beta}
\end{align*}
hold for some constant $C=C(m,\beta)>0$.
%\begin{align*}
%    W_{\beta}(x,t)&\ge \frac{C_0\sqrt{\gamma}}{C_1}t^{-\frac{m}{4}}\int_0^1e^{-\frac{t^{\gamma}}{\gamma}\lambda}\lambda^{\beta-1}d\lambda
%    =\frac{C_0\gamma^{\beta}\sqrt{\gamma}}{C_1}t^{-\frac{m}{4}-\gamma\beta}\int_0^{\frac{1}{\gamma}t^{\gamma}}e^{-s}s^{\beta-1}ds\ge C_4t^{-\frac{m}{4}-\gamma\beta}
%\end{align*}
%hold for any $(x,t)\in \R^N\times[t_0,\infty)$, where $C_4>0$ is a positive constant dependent only on $\beta$ and $m$.

3. Let $(x,t)\in \R^N\times\big(T_0^{\frac{1}{\gamma}},\infty\big)$ with $|x|\le r_0+\gamma^{-1}t^{\gamma}$. Noting that $T_0t^{\gamma}\in (0,1)$, by dividing the integral region $[0,1]$ in the definition of $W_{\beta}$ into $[0,T_0t^{-\gamma})$ and $[T_0t^{-\gamma},\infty)$, the identity
\[
   W_{\beta}(x,t)=\int_0^{T_0t^{-\gamma}}w_{\lambda}(x,t)\lambda^{\beta-1}d\lambda+\int_{T_0t^{-\gamma}}^1w_{\lambda}(x,t)\lambda^{\beta-1}d\lambda
   =:\mathcal{J}^1_{\beta}(x,t)+\mathcal{J}^2_{\beta}(x,t)
\]
hold. By Lemmas \ref{propy} and \ref{Property}, the estimates
\begin{align}
    \mathcal{J}^1_{\beta}(x,t)&\le\int_0^{T_0t^{-\gamma}}e^{\lambda|x|}\lambda^{\beta-1}d\lambda\le \exp(T_0t^{-\gamma}|x|)\int_0^{T_0t^{-\gamma}}\lambda^{\beta-1}d\lambda\notag\\
    &=\beta^{-1}T_0^{\beta}\exp (T_0t^{-\gamma}|x|)t^{-\frac{m+2}{2}\beta}
    \le C t^{-\frac{m+2}{2}\beta}
    \label{estJ_1}
\end{align}
hold for the positive constant $C=C(\beta,m,r_0):=\beta^{-1}T_0^{\beta}\exp(\gamma^{-1}T_0+r_0)>0$.
By the right estimate of (\ref{esty}) and Lemma \ref{Property}, the estimates
\begin{align}
\mathcal{J}_{\beta}^2(x,t)\le Ct^{-\frac{m}{4}}\int_{T_0t^{-\gamma}}^1\left(1+|\lambda x|\right)^{-\frac{N-1}{2}}\exp\left(-\gamma^{-1}\lambda t^{\gamma}+\lambda|x|\right)\lambda^{\beta-\frac{3}{2}+\frac{1}{m+2}}d\lambda
=:Ct^{-\frac{m}{4}}\mathcal{I}_{\beta}(x,t)
\label{2-10}
\end{align}
hold. We note that the estimates $c(t^{\gamma}+\gamma r_0)\le t^{\gamma}\le t^{\gamma}+\gamma r_0$ hold for $t\ge T_0^{\frac{1}{\gamma}}$, where $c=c(T_0):=T_0(T_0+\gamma r_0)^{-1}$ is independent of $t$. If $|x|\le 2^{-1}(r_0+\gamma^{-1}t^{\gamma})$, by noting that the inequalities
\[
   \gamma\left(\gamma^{-1}t^{\gamma}+r_0-|x|\right)\le t^{\gamma}+\gamma r_0\le 2\gamma\left(\gamma^{-1}t^{\gamma}+r_0-|x|\right)
\]
hold and changing variables with $(2\gamma)^{-1}t^{\gamma}\lambda=\mu$, the estimates hold:
\begin{align}
    \mathcal{I}_{\beta}(x,t)
    &\le \exp\left(\frac{r_0}{2}\right)\int_{T_0t^{-\gamma}}^1\exp\left(-\frac{t^{\gamma}}{2\gamma}\lambda\right)\lambda^{\beta+\frac{1}{m+2}-\frac{3}{2}}d\lambda
    =\exp\left(\frac{r_0}{2}\right)(2\gamma t^{-\gamma})^{\beta-\frac{1}{2}+\frac{1}{m+2}}\int_{\frac{T_0}{2\gamma}}^{\frac{t^{\gamma}}{2\gamma}}e^{-\mu}\mu^{\beta+\frac{1}{m+2}-\frac{3}{2}}d\mu\notag\\
    &\le \exp\left(\frac{r_0}{2}\right)(2\gamma)^{\beta-\frac{1}{2}+\frac{1}{m+2}}\Gamma\left(\beta-\frac{1}{2}+\frac{1}{m+2}\right)t^{-\frac{m+2}{2}\beta+\frac{m}{4}}.\label{2-11}
\end{align}
If $2^{-1}(r_0+\gamma^{-1}t^{\gamma})<|x|<r_0+\gamma^{-1}t^{\gamma}$, by changing variables with $\lambda(\gamma^{-1}t^{\gamma}+r_0-|x|)=\mu^{-1}$ and the assumption $\beta > \frac{N}{2}-\frac{1}{m+2}$, the estimates
\begin{align}
    \mathcal{I}_{\beta}(x,t)&\le (2\gamma)^{\frac{N-1}{2}}\exp(r_0)t^{-\frac{\gamma(N-1)}{2}}\int_{T_0t^{-\gamma}}^1\exp\left(-\lambda\left(\gamma^{-1}t^{\gamma}+r_0-|x|\right)\right)\lambda^{\beta-\frac{1}{2}+\frac{1}{m+2}-\frac{N+1}{2}}d\lambda\notag\\
    &=(2\gamma)^{\frac{N-1}{2}}\exp(r_0)t^{-\frac{\gamma(N-1)}{2}}\left(\gamma^{-1}t^{\gamma}+r_0-|x|\right)^{-\beta-\frac{1}{m+2}+\frac{N}{2}}\int_{(\gamma^{-1}t^{\gamma}+r_0-|x|)^{-1}}^{t^{\gamma}T_0(\gamma^{-1}t^{\gamma}+r_0-|x|)^{-1}}\exp(-\mu^{-1})\mu^{-\beta-\frac{1}{m+2}+\frac{N-2}{2}}d\mu\notag\\
    &\le (2\gamma)^{\frac{N-1}{2}}\exp(r_0)t^{-\frac{\gamma(N-1)}{2}}\left(\gamma^{-1}t^{\gamma}+r_0-|x|\right)^{-\beta-\frac{1}{m+2}+\frac{N}{2}}\int_{0}^{\infty}\exp(-\mu^{-1})\mu^{-\beta-\frac{1}{m+2}+\frac{N-2}{2}}d\mu\notag\\
    &=Ct^{-\frac{\gamma(N-1)}{2}}\left(\gamma^{-1}t^{\gamma}+r_0-|x|\right)^{-\beta-\frac{1}{m+2}+\frac{N}{2}}\label{2-12}
\end{align}
hold for some positive constant $C=C(N,m,\beta,r_0)$. By combining the estimates (\ref{estJ_1})-(\ref{2-12}), we obtain the conclusion.

4. Similar as the proof of the estimate \eqref{Pro2.4.3}, we divide $W_{\beta,m}(x,t)$ into two parts as follows:
\[
   W_{\beta}(x,t)=\int_0^{T_0t^{-\gamma}}w_{\lambda}(x,t)\lambda^{\beta-1}d\lambda+\int_{T_0t^{-\gamma}}^1w_{\lambda}(x,t)\lambda^{\beta-1}d\lambda
   =:\mathcal{J}^1_{\beta}(x,t)+\mathcal{J}^2_{\beta}(x,t).
\]
It follows from \eqref{esty} and \eqref{2-2-1} that
\begin{align}
    \mathcal{J}^2_{\beta}(x,t)
    %=\int_{T_0t^{-\gamma}}^{1}w_{\lambda,m}(x,t)\lambda^{\beta-1}d\lambda=\int_{T_0t^{-\gamma}}^{1}y_{\lambda,m}(x,t)\varphi_{\lambda}(x)\lambda^{\beta-1}d\lambda\notag\\
    &\leq C \int_{T_0t^{-\gamma}}^1(\lambda^{\frac 1{\gamma}}t)^{-\frac m4}\exp{(-\gamma^{-1}(\lambda^{\frac 1{\gamma}}t)^{\gamma})}\varphi_{\lambda}(x)\lambda^{\beta-1}d\lambda\notag\\
    &= C \int_{T_0t^{-\gamma}}^1 t^{-\frac m4}\exp{\left(-\frac {\lambda}{\gamma}t^{\gamma}\right)}\int_{-1}^1(1-\theta^2)^{\frac {N-3}2}e^{\theta\lambda |x|}d\theta\lambda^{\beta-\frac 32+\frac 1{m+2}}d\lambda\notag\\
    &\leq Ce^{r_0} t^{-\frac m4}\int_{-1}^1(1-\theta^2)^{\frac {N-3}2}d\theta\int_{T_0t^{-\gamma}}^1 \exp{\lambda\left(\theta |x|-\frac {1}{\gamma}t^{\gamma}-r_0\right)}\lambda^{\beta-\frac 32+\frac 1{m+2}}d\lambda\notag\\
    &= C e^{r_0} t^{-\frac m4}\int_{-1}^1(1-\theta^2)^{\frac {N-3}2}\left(r_0+\frac {1}{\gamma}t^{\gamma}-\theta |x|\right)^{-\beta+\frac 12-\frac 1{m+2}}d\theta\int_{T_0t^{-\gamma}}^{r_0+\frac {1}{\gamma}t^{\gamma}-\theta |x|}e^{-\eta}\eta^{\beta-\frac 32+\frac 1{m+2}}d\eta\notag.
\end{align}
Noting that
$$
\int_{T_0t^{-\gamma}}^{r_0+\frac {1}{\gamma}t^{\gamma}-\theta |x|}e^{-\eta}\eta^{\beta-\frac 32+\frac 1{m+2}}d\eta\leq \int_0^{\infty}e^{-\eta}\eta^{\beta-\frac 32+\frac 1{m+2}}d\eta=\Gamma\left(\beta-\frac 12+\frac 1{m+2}\right),
$$
with $0\leq r_0+\frac{1}{\gamma}t^{\gamma}-\theta |x|<\infty$, we have
\begin{align}
     &\mathcal{J}^2_{\beta}(x,t)\\
     &\leq Ce^{r_0}\Gamma\left(\beta-\frac 12+\frac 1{m+2}\right) t^{-\frac m4}\int_{-1}^1(1-\theta^2)^{\frac {N-3}2}(r_0+\frac {1}{\gamma}t^{\gamma}-\theta |x|)^{-\beta+\frac 12-\frac 1{m+2}}d\theta\notag\\
     &=Ce^{r_0}\Gamma\left(\beta-\frac 12+\frac 1{m+2}\right)2^{N-2} t^{-\frac m4}\int_{0}^1(\tilde{\theta}(1-\tilde{\theta})^{\frac {N-3}2}(r_0+\frac {1}{\gamma}t^{\gamma}+|x|-2\tilde{\theta} |x|)^{-\beta+\frac 12-\frac 1{m+2}}d\tilde{\theta}\notag\\
     &\leq Ce^{r_0}\Gamma\left(\beta-\frac 12+\frac 1{m+2}\right)2^{N-2} t^{-\frac m4}\left(r_0+\frac {1}{\gamma}t^{\gamma}+|x|\right)^{-\beta+\frac 12-\frac 1{m+2}}\int_{0}^1(\tilde{\theta}(1-\tilde{\theta})^{\frac {N-3}2}(1-\tilde{\theta}z)^{-\beta+\frac 12-\frac 1{m+2}}d\tilde{\theta},\notag
\end{align}
where $\tilde{\theta}=\frac {\theta+1}2$ and $z=\frac {2|x|}{r_0+\frac {1}{\gamma}t^{\gamma}+|x|}$.

Direct analysis shows that if $\frac 12 - \frac 1{m+2}<\beta<\frac {N}2 - \frac 1{m+2}$,
$\int_{0}^1(\tilde{\theta}(1-\tilde{\theta})^{\frac {N-3}2}(1-\tilde{\theta}z)^{-\beta+\frac 12-\frac 1{m+2}}d\tilde{\theta}$ is bounded for $z\in [0,1]$. Thus,
\begin{equation}\label{2-31}
\mathcal{J}^2_{\beta}(x,t)\leq C(N,m,\beta,r_0) t^{-\frac m4}\left(r_0+\frac {1}{\gamma}t^{\gamma}+|x|\right)^{-\beta+\frac 12-\frac 1{m+2}},
\end{equation}
which completes the proof by comparing \eqref{2-31} with \eqref{estJ_1}.
\end{proof}

%% file: Proofofthemain.tex
\section{Proof of Theorems \ref{Main} and \ref{Main2}}
\label{Proofmain}
\ \ \ In this section we give proofs of Theorems \ref{Main} and \ref{Main2}. Let us assume all the assumptions of Theorem \ref{Main} except for $m_1\ne m_2$ . We simply write $T(\varepsilon)$ as $T$ if it does not cause a confusion. We denote  by $p'$ and $q'$ the H\"older conjugate of $p$ and $q$ respectively. We only consider the case $m_1\ge m_2$, since the other case $m_1<m_2$ can be treated in the similar manner. In the case $m_1\ge m_2$, we only consider the case $q>2$, since the other case $q\le 2$ can be treated more simply. Let $\lambda_0\in (0,1)$, which will be determined later (see (\ref{estg_1})) and is independent of $\varepsilon$. We introduce $T_2=T_2(\lambda_0,m_1,m_2,N,q,r_0)>0$ given by
\[
   T_2:=2
   \begin{cases}
\ \max\left(\lambda_0^{-\frac{1}{\gamma(m_1)}}T_1(m_1),\left(\frac{m_1+2}{m_2+2}\right)^{\frac{2}{m_1-m_2}}, \{\gamma(m_1)\max(\mathcal{R}_0(N,q',\lambda_0)-r_0,1+r_0)\}^{\frac{1}{\gamma(m_1)}}\right), &\text{if}\ m_1>m_2,\\
\ \max\left(\lambda_0^{-\frac{1}{\gamma(m)}}T_1(m), \{\gamma(m)\max(\mathcal{R}_0(N,q',\lambda_0)-r_0,1+r_0)\}^{\frac{1}{\gamma(m)}}\right), &\text{if}\ m_1=m_2=m
\end{cases}
\]
where $\gamma(m):=\frac{m+2}{2}$, $T_1(m)>0$ is given in Lemma \ref{estdy}, $\mathcal{R}_0$ is given in Corollary \ref{estphi} and $r_0:=\sup \left\{|x|\in \R_{\ge 0}\ |\ x\in \supp(f_1,f_2,g_1,g_2)\right\}$ is given in Theorem \ref{Main}. We note that $T_2$ is independent of $\varepsilon$. If $T\le T_2$, then the conclusions of Theorem \ref{Main} and Theorem \ref{Main2} hold by the assumptions $p,q>1$ and $\Omega_{GG},\Omega_{SS}\ge 0$ with $m_1>m_2$ and $\Gamma_{SS}\ge 0$ with $m_1=m_2$, respectively, and taking $\varepsilon_0$ satisfying $\varepsilon_0\in (0,1]$. Thus we consider the case $T>T_2$ in the following. Let $R\in (T_2,T)$ be a parameter. Then we note that for $t\ge \frac{R}{2}$, the identity $\phi(t)=\frac{1}{\gamma(m_1)}t^{\gamma(m_1)}$ holds. We introduce a test function $\eta\in C_0^\infty([0,\infty))$ satisfying
\begin{equation}
\label{eta}
 \eta(s):=
\begin{cases}
\ 1, & s\in [0,1/2),\\
\ \mbox{decreasing}, &s\in [1/2 ,1),\\
\ 0, &s\in [1,\infty),\\
\end{cases}
\end{equation}
and also define a test function $\eta_R\in C_0^{\infty}([0,\infty))$ given by $\eta_R(s):=\eta\left(\frac{s}{R}\right)$. Let $k\geq \max\{2p',2q'\}$. For $r>0$, we denote by $B(r)$ the closed ball in $\R^N$ centered at the origin with a radius $r$, i.e. $B(r):=\{x\in \R^N\ |\ |x|\le r\}$. We introduce a cut-off function $\chi:\R^N\rightarrow \R_{\ge 0}$ satisfying $\chi\in C_0^{\infty}(\R^N)$ and  $\chi(x)=1$ for any $x\in B(r_0+\gamma(m_1)^{-1} T^{\gamma(m_1)})$.

We introduce a family of test functions $\{\Phi_{\lambda,m_1}(x,t)\}_{\lambda>0}$, whose component $\Phi_{\lambda,m_1}:\R^N\times \R_{\ge 0}\rightarrow \R_{\ge 0}$ is given by $\Phi_{\lambda,m_1}(x,t):=w_{\lambda,m_1}(x,t)\chi(x)\{\eta_R(t)\}^k\in C_0^{\infty}(\R^N\times[0,T))$. Here $w_{\lambda,m_1}(x,t):=y_{\lambda,m_1}(t)\varphi_{\lambda}(x)$ is defined by (\ref{testsub}).

Since for any $\lambda\in (0,1)$, by a direct computation, the equation
\[
   \partial_t\Phi_{\lambda,m_1}(x,t)=\{\eta_R(t)\}^{k-1}\chi(x)\left\{\left(\partial_tw_{\lambda,m_1}(x,t)\right)\eta_R(t)+\frac{k}{R}w_{\lambda,m_1}(x,t)\eta'\left(\frac{t}{R}\right)\right\}
\]
holds for any $(x,t)\in \R^N\times[0,T)$, by the identities $y'_{\lambda,m_1}(t)=\lambda^{\frac{1}{\gamma(m_1)}}y'_1(\lambda^{\frac{1}{\gamma(m_1)}}t)$ and (\ref{y'(0)}),  \begin{equation}
\label{dPhi(0)}
    \lim_{\lambda\rightarrow+0}\partial_t\Phi_{\lambda,m_1}(x,0)=\lim_{\lambda\rightarrow+0} C_0\lambda^{\frac 1\gamma}y'_1(\lambda^{\frac 1\gamma} t)\varphi_\lambda(x) \chi(x) =0
\end{equation} holds for any $x\in \R^N$. For any $\lambda\in (0,1)$, by the identity (\ref{propy0}), the equation $\Phi_{\lambda,m_1}(x,0)=\varphi_{\lambda}(x)\chi(x)$ holds for any $x\in \R^N$. By the assumption on the data $g_1$ and the Lebesgue convergence theorem, the estimates
\begin{align*}
    &\lim_{\lambda\rightarrow+0}\int_{\R^N}g_1(x)\Phi_{\lambda,m_1}(x,0)dx
    =\lim_{\lambda\rightarrow+0}\int_{B(r_0)}g_1(x)\varphi_{\lambda}(x)=\int_{\R^N}g_1(x)dx=I[g_1]>0
\end{align*}
holds. Thus there exists a constant $\lambda_0=\lambda_0(I[g_1])\in (0,1)$ such that the inequality
\begin{equation}
\label{estg_1}
\int_{\R^N}g_1(x)\Phi_{\lambda_0,m_1}(x,0)dx>\frac{1}{2}I[g_1]>0
\end{equation}
holds. By a direct computation and the equation (\ref{freesol1}), the identity
\begin{align}
    &\partial_t^2\Phi_{\lambda_0,m_1}(x,t)-t^{m_1}\Delta \Phi_{\lambda_0,m_1}(x,t)\notag\\
    &=\frac{k}{R}\{\eta_R(t)\}^{k-2}\varphi_{\lambda_0}(x)\left[2(\partial_ty_{\lambda_0,m_1}(t))\eta_R(t)\eta'\left(\frac{t}{R}\right)+\frac{1}{R}y_{\lambda_0,m_1}(t)\left\{(k-1)\left(\eta'\left(\frac{t}{R}\right)\right)^2+\eta_R(t)\eta''\left(\frac{t}{R}\right)\right\}\right]
    \label{Phi_0}
\end{align}
holds for any $(x,t)\in B(r_0+\gamma(m_1)^{-1}T^{\gamma(m_1)})\times [0,T)$. By the identity (\ref{Phi_0}) and Lemma \ref{propy}, there exists a positive constant $C=C(m_1,k)>0$ such that for any $(x,t)\in B(r_0+\gamma(m_1)^{-1}T^{\gamma(m_1)})\times \left(\frac{R}{2},R\right)$, the estimate
\begin{align}
\label{3-5-10}
    \left|\partial_t^2\Phi_{\lambda_0,m_1}(x,t)-t^{m_1}\Delta \Phi_{\lambda_0,m_1}(x,t)\right|
    \le \frac{C}{R}\lambda_0^{\frac{1}{m_1+2}-\frac{1}{2}}(\lambda_0+1)\{\eta_R(t)\}^{k-2}\varphi_{\lambda_0}(x)t^{\frac{m_1}{4}}\exp\left(-\gamma(m_1)^{-1}\lambda_0t^{\gamma(m_1)}\right)
\end{align}
holds. Noting that the estimate $T_2\ge 2\{\gamma(m_1)\max(\mathcal{R}_0(N,q',\lambda_0)-r_0,1+r_0)\}^{\frac{1}{\gamma(m_1)}}$ holds, we can apply Corollary \ref{estphi} with $\mathcal{R}=r_0+\gamma(m_1)^{-1}t^{\gamma(m_1)}$ and $\theta=q'$ to get
\begin{align}
    &\int_{\frac{R}{2}}^R\int_{B(r_0+\gamma^{-1}(m_1)t^{\gamma(m_1)})}\left\{\varphi_{\lambda_0}(x)t^{\frac{m_1}{4}}e^{-\frac {\lambda_0}{\gamma(m_1)} t^{\gamma(m_1)}}\right\}^{q'}dxdt\notag\\
    &\le C\lambda_0^{-\frac{N-1}{2}q'-1}e^{\lambda_0q'r_0}\int_{\frac{R}{2}}^{R}\left(1+r_0+\gamma^{-1}(m_1)t^{\gamma(m_1)}\right)^{(N-1)\left(1-\frac{q'}{2}\right)}t^{\frac{m_1q'}{4}}dt\notag\\
    &\le CR^{\left(\frac{m_1}{4}-\frac{m_1+2}{2}\cdot\frac{N-1}{2}\right)q'+\frac{m_1+2}{2}(N-1)+1},
    \label{3-6-10}
\end{align}
for some positive constant $C=C(N,\lambda_0,m_1,r_0,q)>0$, which is independent of $R$. By the definition of the weak solution to (\ref{weaksol}), the finite propagation speed (\ref{FPS}), the integration by parts, $\supp \eta\subset [0,1]$, $\supp \eta'\subset [1/2,1]$, the estimates (\ref{dPhi(0)}), (\ref{estg_1}), (\ref{3-5-10}) and (\ref{3-6-10}), and $k\ge 2q'$ and the H\"older inequality, the estimates
\begin{equation}\label{equ3}
\begin{aligned}
\frac{\varepsilon}{2}I[g_1]&\le \varepsilon\int_{\R^N}g_1\Phi_{\lambda_0,m_1}(0)dx-\varepsilon\int_{\R^N}f_1\partial_t\Phi_{\lambda_0,m_1}(0)dx+\int_0^R\int_{\R^N}G_1(v(t))\Phi_{\lambda_0,m_1}(t) dxdt\\
&\le \int_0^R\int_{\R^N} u(t) \left\{\partial_t^2\Phi_{\lambda_0,m_1}(t)-t^{m_1} \Delta\Phi_{\lambda_0,m_1}(t)\right\} dxdt\\
&\le \frac{C}{R}\lambda_0^{\frac{1}{m_1+2}-\frac{1}{2}}(\lambda_0+1)\int_{\frac{R}{2}}^R\int_{B(r_0+\gamma^{-1}(m_1)t^{\gamma(m_1)})}|u(t)|\{\eta_R(t)\}^{k-2}\varphi_{\lambda_0}(x)t^{\frac{m_1}{4}}e^{-\frac {\lambda_0}{\gamma(m_1)} t^{\gamma(m_1)}}dxdt\\
&\leq \frac CR \lambda_0^{\frac{1}{m_1+2}-\frac{1}{2}}(\lambda_0+1)\left(\int_{\frac R2}^R\int_{\R^N} |u(t)|^q \{\eta_R(t)\}^{k}dxdt\right)^{\frac 1q}\\
&\ \ \ \ \ \ \ \ \ \ \ \ \ \ \ \ \ \  \times\left(\int_{\frac R2}^R\int_{B(r_0+\gamma^{-1}(m_1)t^{\gamma(m_1)})}\left\{\varphi_{\lambda_0}(x) t^{\frac {m_1}4} e^{-\frac {\lambda_0}{\gamma(m_1)} t^{\gamma(m_1)}}\right\}^{q'} dxdt\right)^{\frac 1{q'}}\\
&\le C\left(\int_{\frac R2}^R\int_{\R^N}G_2(u(t)) \{\eta_R(t)\}^{k}dxdt\right)^{\frac 1q}R^{\gamma_{l}(N,m_1,q)}
\end{aligned}
\end{equation}
hold, where $C>0$ is a constant independent of $R$. Here for $m\in \R$ with $m\ge 0$ and $\theta\in (1,\infty)$, $\gamma_{l}=\gamma_{l}(N,m,\theta)$ is defined by
\[
  \gamma_{l}(N,m,\theta):=\frac m4+\frac{(m+2)(N-1)}{4}-\frac{(m+2)(N-1)}{2\theta}-\frac{1}{\theta}.
\]

%Thus,
%By Corollary \ref{estphi}
%\begin{equation}
%\label{equ4}
%\begin{aligned}
%    &\int_{\frac{R}{2}}^R\int_{B(r_0+\gamma^{-1}(m_1)t^{\gamma(m_1)})}\left\{\varphi_{\lambda_0}(x)t^{\frac{m_1}{4}}e^{-\frac {\lambda_0}{\gamma(m_1)} t^{\gamma(m_1)}}\right\}^{q'}dxdt\\
%    &\le C_6\lambda_0^{-\frac{N-1}{2}q'-1}e^{\lambda_0q'r_0}\int_{\frac{R}{2}}^{R}(1+r_0+\gamma^{-1}(m_1)t^{\gamma(m_1)})^{N-1-\frac{N-1}{2}q'}t^{\frac{m_1q'}{4}}dt\\
%    &\le C_6\lambda_0^{-\frac {N-1}2 q'-1} R^{\frac m4 q'-\frac {m+2}2\frac {N-1}2q'+\frac {m+2}2 N-\frac m2}.
%\end{aligned}
%\end{equation}

%It follows from \eqref{equ3} and \eqref{equ4} that
%\begin{equation}\label{mq2}
%\frac{\varepsilon}{2}I[g_1]\leq C \left(\int_{\frac R2}^R\int_{\R^N} |u|^q \eta_R^{k}dxdt\right)^{\frac 1q} R^{\gamma_{l}(N,m_1,q)},
%\end{equation}
%where
%\[
%  \gamma_{l}(N,m,\theta):=\frac m4-\frac{(m+2)(N-1)}{4}+\frac{(m+2)N-m}{2\theta}-1.
%\]

We define a test function $\Phi:\R^N\times [0,T)\rightarrow \R_{\ge 0}$ given by $\Phi(x,t):=\chi(x)\{\eta_R(t)\}^{k}\in C_0^{\infty}(\R^N\times[0,T))$. Since the identity $\chi(x)=1$ holds for any $x\in B(r_0+\phi(T))$, by the direct computation, the identities
\begin{align*}
    \Delta\Phi(x,t)&=0,\ \ \ \ \partial_t\Phi(x,t)=\frac{k}{R}\{\eta_R(t)\}^{k-1}\eta'\left(\frac{t}{R}\right)\\
    \partial_t^2\Phi(x,t)&=\frac{k}{R^2}\{\eta_R(t)\}^{k-2}\left[(k-1)\left\{\eta'\left(\frac{t}{R}\right)\right\}^2+\eta_R(t)\eta''\left(\frac{t}{R}\right)\right]
\end{align*}
hold for any $(x,t)\in B(r_0+\phi(T))\times[0,T)$. By the identity $\eta'(0)=0$, the equation $\partial_t\Phi(x,0)=0$ holds for any $x\in B(r_0)$. By changing variables and the estimate $R\ge T_2$, the estimates
\begin{align}
    \int_{\frac{R}{2}}^R\int_{B(r_0+\phi(t))}dxdt&=C\int_{\frac{R}{2}}^R\{r_0+\phi(t)\}^Ndt
    \le C\int_{\frac{R}{2}}^Rt^{\gamma(m_1)N}dt\notag\\
    &\le CR^{\gamma(m_1)N+1}
    \label{measu}
\end{align}
hold, where $C=C(r_0,m_1,N)>0$ is a positive constant. Noting that the identity $\Phi(x,0)=1$ holds for any $x\in B(r_0)$ due to $\eta(0)=1$, by the definition of the weak solution (\ref{weaksol}), integration by parts, the inequality $k\ge 2q'$, the H\"older inequality, the assumption (\ref{nonlinearity}) and the inequality (\ref{measu}), the estimates
\begin{align}
    \varepsilon I[g_1]+\int_0^R\int_{B(r_0+\phi(t))}&G_1(v(t))\Phi(x,t)dxdt
    \le\int_0^Ru(t)\partial_t^2\Phi(x,t)dxdt
    \le \frac{C}{R^2}\int_{\frac{R}{2}}^R\int_{B(r_0+\phi(t))}|u(t)|\{\eta_R(t)\}^{\frac{k}{q}}dxdt\notag\\
    &\le \frac{C}{R^2}\left(\int_{\frac{R}{2}}^R\int_{B(r_0+\phi(t))}|u(t)|^q\{\eta_R(t)\}^kdxdt\right)^{\frac{1}{q}}
    \left(\int_{\frac{R}{2}}^R\int_{B(r_0+\phi(t))}dxdt\right)^{\frac{1}{q'}}\notag\\
    &\le CR^{(\gamma(m_1)N+1)\frac{1}{q'}-2}\left(\int_{\frac{R}{2}}^R\int_{B(r_0+\phi(t))}G_2(u(t))\Phi(x,t)dxdt\right)^{\frac{1}{q}}
    \label{basicest}
\end{align}
hold for some positive constant $C=C(r_0,m_1,N,k,q)$. The second inequality of the first line in (\ref{basicest}) comes from the estimate
$$\partial_t \Phi(x,t) \leq \frac C{R^2}\eta_R(t)^{k-2}\leq \frac C{R^2}\eta_R(t)^{\frac kq},$$
for $k\ge 2q'$, since $\eta'\left(\frac{t}{R}\right)$ and $\eta''\left(\frac{t}{R}\right)$ are bounded. In the similar manner as the proof of the estimate (\ref{basicest}), the inequality
\begin{equation}\label{equ2}
    \varepsilon I[g_2]+\int_0^R\int_{B(r_0+\phi(t))}G_2(u(t))\Phi(x,t)dxdt
    \le CR^{(\gamma(m_1)N+1)\frac{1}{p'}-2}\left(\int_{\frac{R}{2}}^R\int_{B(r_0+\phi(t))}G_1(v(t))\Phi(x,t)dxdt\right)^{\frac{1}{p}}
\end{equation}
holds for some positive constant $C=C(r_0,m_1,N,k,p)$.
 By the estimates (\ref{basicest})-(\ref{equ2}) and the assumptions $I[g_1]>0$ and $I[g_2]>0$, the inequalities
\begin{align}
\label{basicest0}
\varepsilon I[g_1]+
    \int_0^R\int_{B(r_0+\phi(t))}G_1(v(t))\Phi(x,t)dxdt&\le CR^{(\gamma(m_1)N+1)\frac{pq-1}{pq}-\frac{2(q+1)}{q}}\left(\int_{\frac{R}{2}}^R\int_{B(r_0+\phi(t))}G_1(v(t))\Phi(x,t)dxdt\right)^{\frac{1}{pq}}\\
    \varepsilon I[g_2]+
    \int_0^R\int_{B(r_0+\phi(t))}G_2(u(t))\Phi(x,t)dxdt&\le CR^{(\gamma(m_1)N+1)\frac{pq-1}{pq}-\frac{2(p+1)}{p}}\left(\int_{\frac{R}{2}}^R\int_{B(r_0+\phi(t))}G_2(u(t))\Phi(x,t)dxdt\right)^{\frac{1}{pq}}\label{3-12}
\end{align}
holds. By combining these estimates and the Young inequality, the estimates
\begin{align}
\label{basicest2}
\varepsilon I[g_1]+\left(1-\frac{1}{pq}\right)\int_0^R\int_{B(r_0+\phi(t))}G_1(v(t))\Phi(x,t)dxdt&\le CR^{-F_{GG}(N,m_1,p,q)}\\
   \varepsilon I[g_2]+\left(1-\frac{1}{pq}\right)\int_0^R\int_{B(r_0+\phi(t))}G_2(u(t))\Phi(x,t)dxdt&\le CR^{-F_{GG}(N,m_1,q,p)}
   \label{basicest3}
\end{align}
holds for any $R\in (T_2,T)$.

\subsection{Proof of Theorem \ref{Main}}
\ \ In this subsection, we complete the proof of Theorem \ref{Main}. By symmetry, we may assume that $m_1>m_2$, which implies that the identities $\Omega_{GG}(N,m_1,m_2,p,q)=\Gamma_{GG}(N,m_1,p,q)$ and $\Omega_{SS}(N,m_1,m_2,p,q)=F_{SS}(N,m,p,q)$ hold. We divide the proof into the three cases (i) $\Gamma_{GG}(N,m_1,p,q)>0$, (ii) $F_{SS}(N,m_1,p,q)>0$, and (iii) $F_{SS}(N,m_1,p,q)=0$.
\subsubsection{$\Gamma_{GG}(N,m_1,p,q)>0$}

\begin{proof}[Proof of Theorem \ref{Main} in the case $\Gamma_{GG}(N,m_1,p,q)>0$ ]
First we consider the case of $p\le q$. Then the identity $\Gamma_{GG}(N,m_1,p,q)=F_{GG}(N,m_1,q,p)$ holds. By the estimate (\ref{basicest3}), the estimate
$$\varepsilon I[g_2]\le R^{-F_{GG}(N,m_1,q,p)}$$
holds for any $R\in (T_2,T)$. By the assumption $F_{GG}(N,m_1,q,p)>0$, the estimate $T(\varepsilon)\le A\varepsilon^{-F_{GG}(N,m_1,q,p)^{-1}}$ holds for some $A>0$ independent of $\varepsilon$. Next we consider the opposite case $p>q$. Then the identity $\Gamma_{GG}(N,m_1,p,q)=F_{GG}(N,m_1,p,q)$ holds. By the estimate (\ref{basicest}), the estimate $\varepsilon I[g_1]\le R^{-F_{GG}(N,m_1,p,q)}$ holds for any $R\in (T_2,T)$. By the assumption $F_{GG}(N,m_1,p,q)>0$, the estimate $T(\varepsilon)\le A\varepsilon^{-F_{GG}(N,m_1,p,q)^{-1}}$ holds. Therefore the estimate $T(\varepsilon)\le A\varepsilon^{-\Gamma_{GG}(N,m_1,p,q)^{-1}}$ holds, which completes the proof of theorem in the case of $\Gamma_{GG}(N,m_1,p,q)>0$
\end{proof}

\subsubsection{$F_{SS}(N,m_1,p,q)>0$}
\label{sub}

\begin{proof}[Proof of Theorem \ref{Main} in the case $F_{SS}(N,m_1,p,q)>0$]
By the estimate (\ref{basicest3}) and the assumption $I[g_2]\ge 0$, the inequality
\begin{equation}
\label{3-15-a}
    \int_0^R\int_{B(r_0+\phi(t))}G_2(u(t))\Phi(x,t)dxdt\le CR^{-F_{GG}(N,m_1,q,p)}
\end{equation}
holds for any $R\in (T_2,T)$. Noting that the identity
\[
    -\frac{1}{q}F_{GG}(N,m_1,q,p)+\gamma_{l}(N,m_1,q)=-F_{SS}(N,m_1,p,q)
\]
holds, by combining the estimates (\ref{equ3}) and (\ref{3-15-a}), the inequality
\begin{equation}
\frac{\varepsilon}{2}I[g_1]
\le C R^{-F_{SS}(N,m_1,p,q)}
\end{equation}
holds for any $R\in (T_2,T)$. By the assumption $F_{SS}(N,m_1,p,q)>0$, the estimate
\begin{equation}
R\leq C \varepsilon^{-F_{SS}(N,m_1,p,q)^{-1}}
\end{equation}
holds for any $R\in (T_2,T)$, which implies that the estimate $T(\varepsilon)\le A\varepsilon^{-F_{SS}(N,m_1,p,q)^{-1}}$ holds for some $A>0$ independent of $\varepsilon$. This completes the proof of the theorem in the case of $F_{SS}(N,m_1,p,q)>0$.
\end{proof}

\subsubsection{$F_{SS}(N,m_1,p,q)=0$ and $N\geq 2$}
\label{Cri}
\ \ In this subsubsection, we consider the critical case $F_{SS}(N,m_1,p,q)=0$ and $N\geq 2$. Then we introduce a test function $L^{\infty}(0,\infty)\ni\eta^*:\R_{\ge 0}\rightarrow \R_{\ge 0} $ given by
\[
   \eta^*(s):=
   \begin{cases}
   0, &s\in [0,1/2),\\
\eta(s), &s\in[1/2,\infty),
   \end{cases}
\]
where $\eta\in C_0^{\infty}([0,\infty))$ is defined by (\ref{eta}). For $\sigma>0$, we also define $ L^{\infty}(0,\infty)\ni \eta_{\sigma}^*:\R_{\ge 0}\rightarrow \R_{\ge 0}$ as $\eta_{\sigma}^*(s):=\eta^*(\frac{s}{\sigma})$ for $s\in [0,\infty)$. To prove the upper estimate (\ref{upper}) of the lifespan in the critical case, we apply the following fact for a system of ordinary differential inequalities (see Lemma 3.10 in \cite{ISW19}):
\begin{lemma}{\cite[Lemma3.10]{ISW19}}
Let $T>2$, $t_0\in (2,T)$, $Y \in C^1([t_0,T))$ be non-negative, $\delta, M_1, M_2>0$ and $p_1, p_2>1$ with $p_2<p_1+1$. We assume that the estimates
\begin{equation}\label{lemma3.1}
\begin{cases}
\delta\leq M_1 t Y'(t), &\quad \quad \quad t \in (t_0,T),\\
Y(t)^{p_1}\leq M_2 t(\log t)^{p_2-1}Y'(t),&\quad \quad \quad t \in (t_0,T),
\end{cases}
\end{equation}
hold. Then there exist positive constants $\delta_0$ and $M_3$ independent of $\delta$ such that the estimate holds:
$$
T\leq \exp\left(M_3 \delta^{-\frac {p_1-1}{p_1-p_2+1}}\right).
$$
\end{lemma}

 For a function $\mathcal{W}\in L_{loc}^1([0,T);L^1(\R^N))$ such that the estimate $\mathcal{W}(x,t)\ge 0$ holds for a.e.\ $(x,t)\in \R^N\times[0,T)$, we introduce a function $Y[\mathcal{W}]:[0,T)\rightarrow \R_{\ge 0}$ defined by
\begin{equation}
\label{YW}
Y[\mathcal{W}](R):=\int_0^R\left[\int_0^T\int_{\R^N}\mathcal{W}(x,t)\{\eta_\sigma^*(t)\}^{k}dxdt\right]\sigma^{-1}d\sigma,\ \ \ \text{for}\ R\in [0,T),
\end{equation}
where $k> 0$. The following estimates for the function $Y[\mathcal{W}]$ were firstly proved in \cite[Proposition 2.1]{IS19}, and are utilized below to prove that the function $\phi:=Y[\mathcal{W}]$ satisfies the estimates (\ref{lemma3.1}) (see also \cite[Lemma3.9]{ISW19}).
\begin{lemma}{\cite[Proposition 2.1]{IS19}}
\label{lemcri}
Let $T>0$, $\mathcal{W}\in L_{loc}^1([0,T);L^1(\R^N))$ such that the estimate $\mathcal{W}(x,t)\ge 0$ holds for a.e. $(x,t)\in \R^N\times [0,T)$, and $Y[\mathcal{W}]:[0,T)\rightarrow \R_{\ge 0}$ be the function given by (\ref{YW}). Then $Y[\mathcal{W}]$ belongs to $C^1((0,T))$ and the estimates
\begin{align*}
    \frac{d}{dR}Y[\mathcal{W}](R)&=R^{-1}\int_0^T\int_{\R^N}\mathcal{W}(x,t)\{\eta_R^*(t)\}^kdxdt,\\
    Y[\mathcal{W}](R)&\le \int_0^T\int_{\R^N}\mathcal{W}(x,t)\{\eta_R^*(t)\}^kdxdt
\end{align*}
hold for any $R\in [0,T)$.
\end{lemma}

\begin{proof}[Proof of Theorem \ref{Main} in the case $F_{SS}(N,m_1,p,q)=0$]
We introduce a function $\mathcal{W}\in L_{loc}^1([0,T);L^1(\R^N))$ given by $\mathcal{W}:=G_1(v)W_{\beta,m_1}$, where $W_{\beta,m_1}$ is given by (\ref{testcri}) with $m=m_1$. From the assumptions of the nonlinear function $G_1$ and the function $v$, we see that the function $\mathcal{W}$ belongs to $L_{loc}^1([0,T);L^1(\R^N))$ and the estimate $\mathcal{W}(x,t)\ge 0$ holds for a.e. $(x,t)\in \R^N\times[0,T)$, which enables us to apply Lemma \ref{lemcri}. We introduce a function $Y_2:[0,T)\rightarrow \R_{\ge 0}$ given by $Y_2(R):=Y[G_1(v)W_{\beta,m_1}](R)$ for $R\in [0,T)$, where $Y[\cdot]$ is given by (\ref{YW}).
By combining the estimates (\ref{equ3}) and (\ref{equ2}) and the assumptions $I[g_1]>0$, $I[g_2]>0$ and $F_{SS}(N,m_1,p,q)=0$, the inequalities
\begin{align*}
    \left(\frac{\varepsilon}{2}I[g_1]\right)^{pq}\le CR^{-\left(\frac{m_1+2}{4}N-\frac{1}{2}\right)+\frac{m_1+2}{2q}}\int_{\frac{R}{2}}^R\int_{B(r_0+\phi(t))}G_1(v(t))\{\eta_R(t)\}^kdxdt
\end{align*}
Thus, applying \eqref{Pro2.4.2} with $\beta=\frac {N}2-\frac 1{m_1+2}-\frac 1q$, we get
\begin{equation}
\begin{aligned}
\left(\frac{\varepsilon}{2}I[g_1]\right)^{pq}
%\leq C \int_0^T\int_{\R^N}G_1(v(t)) \{\eta_R(t)\}^{k}dxdt R^{-\frac {m_1+2}4N+\frac 12+\frac {m_1+2}2\frac 1q}\\
&\leq C \int_0^T\int_{\R^N} G_1(v(t))W_{\beta,m_1}(x,t) \{\eta_R(t)\}^{k}dxdt,
\end{aligned}
\end{equation}
which implies that
\begin{equation}
\varepsilon^{pq}\leq CRY_2'(R).
\end{equation}

Choosing $\Phi_{\beta,m_1}(x,t)=W_{\beta,m_1}(x,t)\chi(x)\eta_R^k(t)$ in the first equation of \eqref{weaksol} and applying integration by parts, we get
\begin{equation}
\label{3-21}
\begin{aligned}
&\varepsilon\int_{\R^N}g_1(x)\Phi_{\beta,m_1}(x,0)dx-\varepsilon\int_{\R^N}f_1(x)\partial_t\Phi_{\beta,m_1}(x,0)dx+\int_0^T\int_{\R^N}G_1(v(t))\Phi_{\beta,m_1} dxdt\\
&\le\int_0^T\int_{\R^N} u (\partial_t^2 \Phi_{\beta,m_1}-t^{m_1} \Delta\Phi_{\beta,m_1}) dxdt\\
&\leq C_1 R^{-1} \int_0^T\int_{\R^N} |u| \left(\exp{(t^{-\gamma}|x|)}t^{-\gamma(\beta+\frac 1\gamma)}+t^{\frac {m_1}2}W_{\beta+1}\right)(\eta_R^*)^{k-2}dxdt+C_2 R^{-2} \int_0^T\int_{\R^N} |u| W_{\beta}(\eta_R^*)^{k-2}dxdt\\
&\leq C R^{-1} \left(\int_0^T\int_{\R^N} |u|^q (\eta_R^*)^{2q'}dxdt\right)^{\frac 1q}\left(\int_{\frac R2}^R\int_{B(r_0+\frac 1\gamma t^\gamma)}\left(\exp{(t^{-\gamma}|x|)}t^{-\gamma(\beta+\frac 1\gamma)}\right)^{q'}dxdt\right)^{\frac 1{q'}}\\
&\quad +C R^{-1} \left(\int_0^T\int_{\R^N} |u|^q (\eta_R^*)^{2q'}dxdt\right)^{\frac 1q}\left(\int_{\frac R2}^R\int_{B(r_0+\frac 1\gamma t^\gamma)}\left(t^{\frac {m_1}2}W_{\beta+1}\right)^{q'}dxdt\right)^{\frac 1{q'}}\\
&\quad +C R^{-2} \left(\int_0^T\int_{\R^N} |u|^q (\eta_R^*)^{2q'}dxdt\right)^{\frac 1q}\left(\int_{\frac R2}^R\int_{B(r_0+\frac 1\gamma t^\gamma)}W_{\beta}^{q'}dxdt\right)^{\frac 1{q'}}
\end{aligned}
\end{equation}
by \eqref{2-19}.

Direct calculation gives
\begin{align}
    \int_{\frac R2}^R\int_{B(r_0+\frac 1\gamma t^\gamma)}\left(\exp{(t^{-\gamma}|x|)}t^{-\gamma(\beta+\frac 1\gamma)}\right)^{q'}dxdt
    &\leq C \int_{\frac R2}^R\int_{B(r_0+\frac 1\gamma t^\gamma)}t^{-\gamma(\beta+\frac 1\gamma)q'}dxdt\notag\\
    &\leq C R^{-(\frac {m_1+2}2 \beta +1)q'+\frac {m_1+2}2 N+1}\notag\\
    &\leq C R^{(\frac {m_1}4-\frac {N-1}2\frac {m_1+2}2)q'+\frac{m_1+2}2 (N-1)+1} \log(R),\notag
\end{align}
for $\beta=\frac {N}2-\frac 1{m_1+2}-\frac 1q$. It follows from \eqref{Pro2.4.3} that
$$
\begin{aligned}
&\int_{\frac R2}^R\int_{B(r_0+\frac 1\gamma t^\gamma)}\left(t^{\frac {m_1}2}W_{\beta+1}\right)^{q'}dxdt\\
&\leq C \int_{\frac R2}^R\int_{B(r_0+\frac 1\gamma t^\gamma)}\left(t^{\frac {m_1}2}\cdot t^{-\frac {m_1}4-\frac {(N-1)(m_1+2)}4}(r_0+\frac 1\gamma t^\gamma-|x|)^{-\beta-\frac 1{m_1+2}+\frac {N-2}2}\right)^{q'}dxdt\\
&\leq C \int_{\frac R2}^R t^{(\frac {m_1}4-\frac {(N-1)(m_1+2)}4)q'}\int_{B(r_0+\frac 1\gamma t^\gamma)}\left(1+\frac 1\gamma t^\gamma-|x|\right)^{(-\beta-\frac 1{m_1+2}+\frac {N-2}2)q'}dxdt\\
\end{aligned}
$$
Noting that $(-\beta-\frac 1{m_1+2}+\frac {N-2}2)q'=-1$, since $\beta=\frac {N}2-\frac 1{m_1+2}-\frac 1q$, we have
$$
\begin{aligned}
&\int_{\frac R2}^R\int_{B(r_0+\frac 1\gamma t^\gamma)}\left(t^{\frac {m_1}2}W_{\beta+1}\right)^{q'}dxdt\\
&\leq C \int_{\frac R2}^R t^{(\frac {m_1}4-\frac {N-1}2\frac {m_1+2}2)q'}\int_{B(r_0+\frac 1\gamma t^\gamma)}\left(1+\frac 1\gamma t^\gamma-|x|\right)^{-1}dxdt\\
&\leq C R^{(\frac {m_1}4-\frac {N-1}2\frac {m_1+2}2)q'+\frac{m_1+2}2 (N-1)+1} \log(R).
\end{aligned}
$$

Similarly, applying \eqref{Pro2.4.4}, we could obtain
\begin{align*}
    &\int_{\frac R2}^R\int_{B(r_0+\frac 1\gamma t^\gamma)}\left(W_{\beta}\right)^{q'}dxdt\\
    &\leq C \int_{\frac R2}^R\int_{B(r_0+\frac 1\gamma t^\gamma)} t^{-\frac {m_1}4 q'}(r_0+\frac 1\gamma t^{\gamma}+|x|)^{(-\beta+\frac 12-\frac 1{m_1+2})q'}dxdt\\
    &\leq C \int_{\frac R2}^R t^{-\frac {m_1}4 q'}(r_0+\frac 1\gamma t^{\gamma})^{-\frac {N-3}2 q'}dt \int_{B(r_0+\frac 1\gamma t^\gamma)} (r_0+\frac 1\gamma t^{\gamma}+|x|)^{-1}dx\\
    &\leq C \int_{\frac R2}^R t^{-\frac {m_1}4 q'}(r_0+\frac 1\gamma t^{\gamma})^{-\frac {N-3}2 q'} t^{\frac {m_1+2}2 (N-1)} \log(t)dt \\
    &\leq C R^{(\frac {m_1}4-\frac {N-1}2\frac {m_1+2}2)q'+\frac{m_1+2}2 (N-1)+1+q'} \log(R)
\end{align*}

Since $\lim_{\lambda\rightarrow 0}\sup_{x\in B(r_0)} \varphi_\lambda(x) =1$, there exists a constant $0<\kappa_0<<1$ such that $\varphi_\lambda(x)\sim 1$ for $\lambda\in [0,\kappa_0]$.
Thus, $\Phi_\beta(x,0)=W_\beta(x,0)=\int_0^{\kappa_0}\varphi_\lambda(x)\lambda^{\beta-1}d\lambda \sim \int_0^{\kappa_0}\lambda^{\beta-1}d\lambda$, which is a constant. It follows from the assumption $\int_{\R^N}g_1(x)dx>0$ that $\int_{\R^N}g_1(x)\Phi_\beta(x,0)dx>0$. Noting that $\partial_t\Phi_\beta(x,0)=0$, we have
$$
\begin{aligned}
&\left(\int_0^T\int_{\R^N}G_1(v(t))\Phi_\beta dxdt\right)^{pq}\\
&\leq C \left(\int_0^T\int_{\R^N} G_2(u(t)) (\eta_R^*)^{2q'}dxdt\right)^{p}\cdot R^{(\frac {m_1}4-\frac {N-1}2\frac {m_1+2}2)pq+(\frac{m_1+2}2 (N-1)+1)\frac {pq}{q'}-pq} (\log(R))^{\frac {pq}{q'}}\\
&\leq C \int_0^T\int_{\R^N} G_1(v(t)) (\eta_R^*)^{2p'}dxdt\cdot R^{(\frac {m_1}4-\frac {N-1}2\frac {m_1+2}2)pq+(\frac{m_1+2}2 (N-1)+1)\frac {pq}{q'}-pq+(\frac {m_1+2}2N-1)(p-1)-2} (\log(R))^{\frac {pq}{q'}}
\end{aligned}
$$
Since $F_{SS}(N,m_1,p,q)=-\frac {m_1+2}4-\frac {m_1}4-\frac {m_1}2\frac 1q+(pq-1)^{-1}(p+2+\frac 1q)=0$, direct calculation gives
$$
\begin{aligned}
&\left(\frac {m_1}4-\frac {N-1}2\frac {m_1+2}2\right)pq+(\frac{m_1+2}2 (N-1)+1)\frac {pq}{q'}-pq+(\frac {m_1+2}2N-1)(p-1)-2 \\
&=-\frac {m_1+2}2 \beta,
\end{aligned}
$$
with $\beta=\frac {N}2-\frac 1{m_1+2}-\frac 1q$. Therefore, by \eqref{Pro2.4.2},
$$
\begin{aligned}
&\left(\int_0^T\int_{\R^N}G_1(v(t))\Phi_\beta dxdt\right)^{pq}\\
&\leq C \int_0^T\int_{\R^N} G_1(v(t)) (\eta_R^*)^{2p'}dxdt\cdot R^{-\frac {m_1+2}2 \beta} (\log(R))^{p(q-1)}\\
&\leq C \int_0^T\int_{\R^N} G_1(v(t)) W_\beta (\eta_R^*)^{2p'}dxdt (\log(R))^{p(q-1)},
\end{aligned}
$$
that is
$$
\left(Y_2(R)\right)^{pq} \leq CR (\log(R))^{p(q-1)}Y_2'(R).
$$

Applying Lemma 3.1 with $\delta=\varepsilon^{pq}, p_1=pq, p_2=p(q-1)+1$, we obtain
$$
T(\varepsilon)\leq \exp (C \varepsilon^{-q(pq-1)}).
$$

\end{proof}

\subsection{Proof of Theorem \ref{Main2}}
\label{section4}
In the case that $m_1=m_2=m$, similar as in the proof of Theorem \ref{Main}, we could obtain
\begin{equation}
T(\varepsilon)\leq C \varepsilon^{-F_{SS}(N,m,q,p)^{-1}}
\end{equation}
if $F_{SS}(N,m,q,p)>0$,
and
$$
T(\varepsilon)\leq \exp (C \varepsilon^{-p(pq-1)}).
$$
if $F_{SS}(N,m,q,p)=0$ and $N\geq 2$. Thus, we have
\begin{equation}
T(\varepsilon)\leq \left\{
\begin{array}{ll}
C \varepsilon^{-\Gamma_{SS}(N,m,p,q)^{-1}},&\quad \mbox{if}\ \  \Gamma_{SS}(N,m,p,q)>0,\\
\exp (C \varepsilon^{-\min\{p(pq-1),q(pq-1)\}}),&\quad \mbox{if}\ \  \Gamma_{SS}(N,m,p,q)=0,\ p\neq q,\ \mbox{and}\ N\geq 2\\
\end{array}
\right.
\end{equation}
So, in this section, we just discuss the case that $\Gamma_{SS}(N,m,p,q)=0,\ p= q$ and $N\geq 2$. In this case, the assumption $\Gamma_{SS}(N,m,p,q)=0$ implies that the identities $ p=q=\rho_{c}$ hold, where $\rho_c=\rho_c(N,m)$ is defined as the positive root of the quadratic equation (\ref{criticalexp}). Let $\Phi\in C_0^{\infty}(\R^N\times[0,T))$. Since a pair of functions $(u,v)$ is a weak solution to the problem (\ref{sys}), by the assumption (\ref{nonlinearity}) and the elementary inequality $a^p+b^p\ge 2^{-p}(a+b)^p$ for $a,b\ge 0$, the estimates
\begin{align}
&-\int_0^T\int_{\R^N} \partial_t(u(t)+v(t))\partial_t\Phi(x,t)dxdt+\int_0^T\int_{\R^N}t^m \nabla (u(t)+v(t))\cdot \nabla \Phi(x,t) dxdt\notag\\
&\ge \varepsilon\int_{\R^N}(g_1+g_2)\Phi(0)dx+\int_0^T\int_{\R^N}(|v(t)|^p+|u(t)|^p)\Phi(x,t) dxdt\notag\\
&\geq \varepsilon\int_{\R^N}(g_1+g_2)\Phi(0)dx+2^{-p}\int_0^T\int_{\R^N}|v(t)+u(t)|^p\Phi dxdt\label{add}
%&\geq \varepsilon\int_{\R^N}(g_1(x)+g_2(x))\Phi(x,0)dx+2^{-p}\int_0^T\int_{\R^N}|v+u|^p\Phi dxdt\\
\end{align}
hold. We define the functions $f\in C_0^{\infty}(\R^N)$, $g\in C_0^{\infty}(\R^N)$ and $z\in L^{p}_{loc}\left(\R^N\times(0,T)\right)$ given by $f:=f_1+f_2$, $g:=g_1+g_2$ and $z(t):=u(t)+v(t)$. Then the estimate \eqref{add} gives
\begin{equation}\label{w}
\varepsilon\int_{\R^N}g\Phi(0)dx+2^{-p}\int_0^T\int_{\R^N}|z(t)|^p\Phi(x,t) dxdt
\leq -\int_0^T\int_{\R^N} \partial_tz(t)\partial_t\Phi(x,t) dxdt+\int_0^T\int_{\R^N}t^m \nabla z(t)\cdot \nabla \Phi(x,t) dxdt.
\end{equation}

In the similar manner as the proof of the estimate \eqref{equ3}, by the estimate (\ref{add}) with $\Phi=\Phi_{\lambda_0}$, the estimate
\begin{equation}\label{equ5}
\varepsilon\int_{\R^N}g\Phi_{\lambda_0}(0)dx\leq C \left(\int_{0}^R\int_{\R^N} |z(s)|^p \{\eta_R(t)\}^{k}dxdt\right)^{\frac 1p} R^{\gamma_l(N,m,p)}
\end{equation}
holds for any $R\in (T_2,T)$.
Set $Y_1(R):=Y[|z|^pW_{\beta}](R)$.
Thus, by \eqref{equ5} and Proposition 2.4 with $\beta=-\frac {2m}{2(m+2)}-\frac m{2(m+2)}p-\frac {N-1}2 p+N$,
\begin{equation}
\begin{aligned}
Y_1'(R)R&=\int_0^T\int_{\R^N}|z|^p W_\beta (\eta_R^*)^{2p'}dxdt\\
&\geq C R^{-\frac {m+2}2 \beta}\int_0^T\int_{\R^N}|z|^p(\eta_R^*)^{2p'}dxdt\\
&\geq C R^{-\frac {m+2}2 \beta}\cdot R^{-\frac m4 p-\frac {m+2}2\frac {N-1}2 p+\left(\frac {m+2}2 N-\frac m2\right)}\left(\varepsilon\int_{\R^N}g(x)\varphi(x,0)dx\right)^p\\
&= C \left(\varepsilon\int_{\R^N}g(x)\varphi(x,0)dx\right)^p.
\end{aligned}
\end{equation}

To get the second inequality of Lemma 3.1 \eqref{lemma3.1}, we substitute $\Phi_\beta(x,t)=W_\beta(x,t)\chi(x)\eta_R^k(t)$ as the test function in \eqref{w}. Then integration by parts gives
\begin{equation}\label{equ6}
\begin{aligned}
&\varepsilon\int_{\R^N}g(x)\Phi_\beta(x,0)dx-\varepsilon\int_{\R^N}f(x)\partial_t\Phi_\beta(x,0)dx+\int_0^T\int_{\R^N}|z|^p\Phi_\beta dxdt\\
&\leq \int_0^T\int_{\R^N} z (\partial_t^2 \Phi_\beta-t^m \Delta\Phi_\beta) dxdt\\
&\leq C_1 R^{-1} \int_0^T\int_{\R^N} |z| \left(\exp{(t^{-\gamma}|x|)}t^{-\gamma(\beta+\frac 1\gamma)}+t^{\frac {m}2}W_{\beta+1}\right)(\eta_R^*)^{k-2}dxdt+C_2 R^{-2} \int_0^T\int_{\R^N} |z| W_{\beta}(\eta_R^*)^{k-2}dxdt\\
&\leq C R^{-1} \left(\int_0^T\int_{\R^N} |z|^p (\eta_R^*)^{2p'}dxdt\right)^{\frac 1p}\left(\int_{\frac R2}^R\int_{B(r_0+\frac 1\gamma t^\gamma)}\left(\exp{(t^{-\gamma}|x|)}t^{-\gamma(\beta+\frac 1\gamma)}\right)^{p'}dxdt\right)^{\frac 1{p'}}\\
&\quad C R^{-1} \left(\int_0^T\int_{\R^N} |z|^p (\eta_R^*)^{2p'}dxdt\right)^{\frac 1p}\left(\int_{\frac R2}^R\int_{B(r_0+\frac 1\gamma t^\gamma)}\left(t^{\frac m2}W_{\beta+1}\right)^{p'}dxdt\right)^{\frac 1{p'}}\\
&\quad +C R^{-2} \left(\int_0^T\int_{\R^N} |z|^p (\eta_R^*)^{2q'}dxdt\right)^{\frac 1p}\left(\int_{\frac R2}^R\int_{B(r_0+\frac 1\gamma t^\gamma)}W_{\beta}^{p'}dxdt\right)^{\frac 1{p'}}
\end{aligned}
\end{equation}
by \eqref{2-19}. Noting that $\beta=-\frac {2m}{2(m+2)}-\frac m{2(m+2)}p-\frac {N-1}2 p+N$ and $p=\rho_{c}(N,m)$, we could obtain
$$
\int_{\frac R2}^R\int_{B(r_0+\frac 1\gamma t^\gamma)}\left(\exp{(t^{-\gamma}|x|)}t^{-\gamma(\beta+\frac 1\gamma)}\right)^{p'}dxdt\leq C R^{(\frac m4-\frac {N-1}2\frac {m+2}2)p'+\frac{m+2}2 (N-1)+1} \log(R),
$$

$$
\int_{\frac R2}^R\int_{B(r_0+\frac 1\gamma t^\gamma)}\left(t^{\frac m2}W_{\beta+1}\right)^{p'}dxdt\leq C R^{(\frac m4-\frac {N-1}2\frac {m+2}2)p'+\frac{m+2}2 (N-1)+1} \log(R),
$$
and
$$
\int_{\frac R2}^R\int_{B(r_0+\frac 1\gamma t^\gamma)}\left(W_{\beta}\right)^{p'}dxdt\leq C R^{(\frac m4-\frac {N-1}2\frac {m+2}2)p'+\frac{m+2}2 (N-1)+1+p'} \log(R).
$$

Since $\int_{\R^N}g_1(x)\Phi_\beta(x,0)dx>0$ and $\partial_t\Phi_\beta(x,0)=0$, we get
$$
\begin{aligned}
&\left(\int_0^T\int_{\R^N}|z|^p\Phi_\beta dxdt\right)^{p}\\
&\leq C \int_0^T\int_{\R^N} |z|^p (\eta_R^*)^{2p'}dxdt\cdot R^{(\frac m4-\frac {N-1}2\frac {m+2}2)p+(\frac{m+2}2 (N-1)+1)\frac {p}{p'}-p} (\log(R))^{\frac {p}{p'}}\\
&\leq C \int_0^T\int_{\R^N} |z|^p W_\beta (\eta_R^*)^{2p'}dxdt (\log(R))^{p-1},
\end{aligned}
$$
by applying \eqref{Pro2.4.2}, since
$$
(\frac m4-\frac {N-1}2\frac {m+2}2)p+(\frac{m+2}2 (N-1)+1)\frac {p}{p'}-p=-\frac {m+2}2 \beta
$$
with $\beta=-\frac {2m}{2(m+2)}-\frac m{2(m+2)}p-\frac {N-1}2 p+N$.

That is
$$
\left(Y_1(R)\right)^{p} \leq CR (\log(R))^{(p-1)}Y_1'(R).
$$

Applying Lemma 3.1 with $\delta=\varepsilon^{p}, p_1=p_2=p$, we obtain
$$
T(\varepsilon)\leq \exp (C \varepsilon^{-p(p-1)}).
$$